\documentclass[11pt,a4paper]{article}

\usepackage[T1]{fontenc}
\usepackage[utf8]{inputenc}

\usepackage{amsfonts, amsmath, amssymb, amsthm, dsfont, mathrsfs, latexsym, euscript, subfigure}
\usepackage[nottoc]{tocbibind}

\usepackage{csquotes}
\usepackage[hyphens]{url}
\usepackage{hyperref}
\usepackage{etoolbox}
\apptocmd{\bibsetup}{\sloppy}{}{}

\usepackage[
  backend=biber,
  style=ieee,
  sorting=nyt,
  url=true
]{biblatex}
\usepackage{tikz}
\usepackage{xcolor, ulem}
\usepackage{array,multirow,makecell}
\usepackage{caption, subcaption}
\usepackage[left=3cm,right=3cm,top=3cm,bottom=3cm]{geometry}

\def\CC{\mathbb{C}}

\def\EE{\mathbb{E}}
\def\FF{\mathbb{F}}
\def\GG{\mathbb{G}}

\def\LL{\mathbb{L}}

\def\NN{\mathbb{N}}

\def\PP{\mathbb{P}}

\def\RR{\mathbb{R}}

\def\Bcal{\mathcal{B}}
\def\Ccal{\mathcal{C}}

\def\Fcal{\mathcal{F}}
\def\Gcal{\mathcal{G}}

\def\Ncal{\mathcal{N}}

\def\Tcal{\mathcal{T}}

\def\Wcal{\mathcal{W}}

\def\Ycal{\mathcal{Y}}

\def\Jtil{\widetilde{J}}
\def\Ntil{\widetilde{N}}

\newcommand{\eps}{\varepsilon}
\def\E#1{\mathbb{E}\left[ #1 \right]}

\def\ind#1{\mathds{1}_{\{#1\}}}
\def\1{\mathds{1}}
\def\tilde#1{\widetilde{#1}}
\def\hat#1{\widehat{#1}}

\newtheorem{theorem}{Theorem}[section]
\makeatletter
\renewcommand{\thetheorem}{%
  \ifnum\c@subsection=0
    \thesection.\number\c@theorem
  \else
    \thesubsection.\number\c@theorem
  \fi}
\makeatother

\newtheorem{cor}[theorem]{Corollary}
\newtheorem{lem}[theorem]{Lemma}
\newtheorem{prop}[theorem]{Proposition}
\newtheorem{definition}[theorem]{Definition}
\newtheorem{assumption}[theorem]{Assumption}
\newtheorem{notation}[theorem]{Notation}
\newtheorem{remark}[theorem]{Remark}

\numberwithin{equation}{section}
\makeatletter
\renewcommand{\theequation}{%
  \ifnum\c@subsection=0
    \thesection.\number\c@equation
  \else
    \thesubsection.\number\c@equation
  \fi}
\makeatother

\def\i{\mathbf{i}}
\def\bar#1{\overline{#1}}
\DeclareSymbolFontAlphabet{\mathrsfs}{rsfs}
\DeclareMathOperator*{\esssup}{ess\,sup}

\title{Quantification of limit theorems for Hawkes processes}
\author{Laure Coutin\footnote{UPS, IMT UMR CNRS 5219, Universit\'e de Toulouse, 135 avenue de Rangueil 31077 Toulouse Cedex 4 France. \; Email: \texttt{laure.coutin@math.univ-toulouse.fr}} \and Benjamin Massat\footnote{UPS, IMT UMR CNRS 5219, Universit\'e de Toulouse, 135 avenue de Rangueil 31077 Toulouse Cedex 4 France. \; Email: \texttt{benjamin.massat@math.univ-toulouse.fr}} \and Anthony Réveillac\footnote{INSA de Toulouse, IMT UMR CNRS 5219, Universit\'e de Toulouse, 135 avenue de Rangueil 31077 Toulouse Cedex 4 France. \; Email: \texttt{anthony.reveillac@insa-toulouse.fr}} }

\begin{document}
\maketitle

\begin{abstract}
In this article, we fill a gap in the literature regarding quantitative functional central limit theorems (qfCLT) for Hawkes processes by providing an upper bound for three limit theorems involving Nearly Unstable Hawkes Process (NUHP), Weakly Critical Hawkes Process (WCHP) and Supercritical Nearly Unstable Hawkes Process (SNUHP). Note that, for two of them, no speed of convergence has been established even for one-dimensional marginals; we provide in this paper a control in terms of a supremum norm in $2$-Wasserstein distance. To do so, we make use of the so-called Poisson imbedding representation and provide a qfCLT formulation in terms of a Brownian sheet. Incidentally, we construct an optimal coupling between a rescaled compensated two-parameter Poisson random measure and a Brownian sheet with respect to the $2$-Wasserstein distance and analyze the asymptotic quality of this coupling in detail.
\end{abstract}



\section{Introduction}
Introduced by Hawkes in \cite{hawkes_spectra_1971}, the Hawkes process $H$ is a self-exciting point process whose intensity, denoted by $\lambda$, is given by  
\begin{equation}
    \lambda_t = \mu + \int_{(0,t)} \phi(t-s) dH_s, \quad t > 0,
\end{equation}
where $\mu > 0$ and $\phi: \RR_+ \to \RR_+$ is a kernel function. While the linear Hawkes process was the first to be studied, a nonlinear version can also be defined. Over the past two decades, the limiting behavior of such processes has been widely investigated.

\subsection{Limits theorems for Hawkes processes}
In 2013, Bacry, Delattre, Hoffmann, and Muzy \cite{bacry_limit_2013} established a law of large numbers and a functional central limit theorem (fCLT) for linear multivariate Hawkes processes observed over a time interval $[0,T]$ as $T \to +\infty$:
\begin{equation*}
    \left(\frac{H_{tT}- \int_0^{tT} \lambda_s ds}{\sqrt{T}}\right)_{t\in[0,1]} \xrightarrow[]{T\to +\infty} \left( \sigma B_t\right)_{t\in [0,1]},
\end{equation*}
where $B_t$ is a Brownian motion. This result was later extended by Zhu \cite{zhu_nonlinear_2013} to nonlinear Hawkes processes.

The first bound on the convergence rate of such processes was established by Torrisi \cite{torrisi_gaussian_2016} for nonlinear Hawkes processes. However, his bound does not tend to zero as $T \to +\infty$.  A major issue is the quantification of functional Central Limit Theorems; this is adresses in \cite{besancon_diffusive_2024} by Besancon, Coutin and Decreusefond using other techniques. Besides, Besançon, Coutin, Decreusefond, and Moyal, in \cite{besancon_diffusive_2024}, quantified convergences to diffuse limits of Markov processes and long-time behavior of Hawkes processes. In this paper, they study the convergence of functionals of a one-dimensional compensated Poisson measure towards a functional of a Brownian motion. For the Hawkes process, they establish the convergence of the linearly interpolated renormalized Hawkes process towards the linearly interpolated Brownian motion.
Between 2018 and 2022, a series of works by Hillairet, Huang, Khabou, Privault, and R\'eveillac (see \cite{hillairet_malliavin-stein_2022}, \cite{khabou_malliavin-stein_2021}, \cite{privault_stein_2018}) analyzed the convergence rate of the linear Hawkes process in 1-Wasserstein distance of one-dimensional marginals, obtaining the bound  
\begin{equation}
\label{eq: maj_intro}
    \Wcal_1 \left( \frac{H_T - \int_0^T \lambda_t dt}{\sqrt{T}}, \Ncal\left(0,\sigma^2 \right)\right) \leq \frac{C}{\sqrt{T}}.
\end{equation}
Their work relies on the Poisson imbedding representation coupled with a taylor-made Malliavin calculus for Hawkes processes completing the approach of \cite{torrisi_gaussian_2016} for this class of processes. Recall that Poisson imbedding consists in representing a counting process as a solution of an SDE driven by a Poisson random measure. We refer to Brémaud and Massoulié \cite{bremaud_stability_1996} for a discussion on Poisson imbedding and to Peccati, Solé, Taqqu, and Utzet \cite{peccati_steins_2010} for Malliavin calculus on the Poisson space.

This methodology is also used in \cite{coutin_normal_2024}, where Coutin, Massat, and Réveillac extend the previous results by obtaining new upper bounds in Wasserstein distance between a functional of point processes and a Gaussian distribution. Their results apply to nonlinear Hawkes processes and discrete linear Hawkes processes (see \cite{kirchner_hawkes_2016} and \cite{quayle_etude_2022} for details on these processes). They also improve the convergence theorem for nonlinear Hawkes processes from \cite{zhu_nonlinear_2013} by relaxing some assumptions.

Furthermore, Horst and Xu \cite{horst_functional_2024} established functional and scaling limit theorems for Hawkes processes under minimal conditions on their kernels using Fourier analysis. In particular, they characterized the limiting behavior of subcritical ($ \left\| \phi \right\|_1 < 1$) and critical ($\left\| \phi \right\|_1 = 1$) Hawkes processes. More recently, Xu \cite{xu_scaling_2024} also studied the scaling limit of multivariate Hawkes processes.

This critical behavior has also been studied by Jaisson and Rosenbaum in \cite{jaisson_limit_2015}. Specifically, the authors investigate the Nearly Unstable Hawkes Process (NUHP), a class of Hawkes processes whose kernel functions depend on $T$ and whose norm, while still less than $1$, approaches the critical value $1$. In particular, the intensity of an NUHP can be written as:  
\begin{equation}
    \lambda^T_t = \mu + \int_{(0,t)} \phi^T(t-s) dH^T_s, \quad t\in [0,T], \quad \text{with } \left\| \phi^T \right\|_1 \xrightarrow[]{T\to +\infty} 1.
\end{equation}
In \cite{jaisson_limit_2015}, Jaisson and Rosenbaum proved that the rescaled intensity $\left(\lambda^T_{tT}/T\right)_{t\in [0,1]}$ converges in law, in the Skorokhod topology, to a CIR process (see Theorem \ref{thm: JR}). While \cite{coutin_normal_2024} provides the first bound in the literature for NUHPs, their work does not establish a convergence rate for these processes.  
More recently, Liu, Xu, and Zhang \cite{liu_scaling_2024}, inspired by \cite{jaisson_limit_2015}, studied kernels whose norm exceeds $1$ but also converges to $1$. This Hawkes process is referred to as Supercritical Nearly Unstable Hawkes Process (SNUHP) by the authors. They prove a convergence theorem, in the Skorokhod topology, for the rescaled counting process $\left(H^T_{tT}/T^2\right)_{t\in [0,1]}$, which converges to the solution of an SDE closely related to the CIR process. The convergence of the rescaled intensity is proved in this article.\\
Since the convergence theorems in \cite{jaisson_limit_2015} and in \cite{liu_scaling_2024} relies on the convergence of a rescaled Poisson measure to a Gaussian measure, we focus on this convergence. By adapting the proof of Fomichov, Gonz\'alez C\'azares, and Ivanovs \cite{fomichov_implementable_2021}, we construct an optimal coupling between a rescaled compensated two-parameter Poisson measure and a two-parameter Gaussian measure. In particular, we develop a comonotonic coupling that allows us to control differences of stochastic integrals. It yields an upper bound on the mean squared maximal distance between paths. Moreover, we also use our methodology to improve the convergence rate for Weakly Critical Hawkes Process (WCHP) presented in \cite{horst_functional_2024}.

\subsection{Our contribution}
As previously mentioned, we focus on three distinct Hawkes process regimes: the Nearly Unstable Hawkes Process (NUHP - $\left\| \phi^T\right\|_1 \nearrow 1$ as $T\to +\infty$), the Weakly Critical Hawkes Process (WCHP - $\left\| \phi^T\right\|_1 = 1$ ), and the Supercritical Nearly Unstable Hawkes Process (SNUHP - $\left\| \phi^T\right\|_1 \searrow 1$ as $T\to +\infty$). For each of these models, we consider the renormalized intensity process defined by
$$\Lambda^T_{\cdot} := \frac{\lambda^T_{T \cdot}}{T},$$
and we denote the corresponding processes by $\Lambda^{T,-}$, $\Lambda^{T,0}$, and $\Lambda^{T,+}$ for the NUHP, WCHP, and SNUHP cases, respectively. Their associated limiting processes are denoted by $X^-$, $X^0$, and $X^+$.\\
When the discussion applies uniformly along the three regimes, we use the general notation $\Lambda^{T,\natural}$ and $X^\natural$, where $\natural \in \{-, 0, +\}$ (see \eqref{eq:Lambda_intro} and \eqref{def: CIR} below for exact formulas). Our contribution will be detailed in two main questions:\\
\begin{center}
    \textit{How can quantification of fCLTs for near-critical Hawkes processes be addressed and in which norm?}
\end{center}
It is important to note that although fCLTs have been established in \cite{jaisson_limit_2015} for the NUHP, in \cite{horst_functional_2024} for the WCHP, and in \cite{liu_scaling_2024} for the SNUHP, only the result in \cite{horst_functional_2024} provides a quantitative convergence rate, specifically for the WCHP but not in the uniform norm (see Remark \ref{rem: horst_xu}). For the NUHP and SNUHP, no such quantitative estimates were previously available. We summarize in Table \ref{tab:int_conv_result} the state of the art and underline our results.
\renewcommand{\arraystretch}{1.5}
\newcolumntype{C}[1]{>{\centering\arraybackslash }m{#1}}
\begin{table}[!ht]
    \centering
    \hspace*{-0.5cm}\begin{tabular}{|C{1.5cm}|C{6cm}|C{7cm}|}
    \cline{2-3} \multicolumn{1}{c|}{}  & $\Lambda^{T,\natural} \xrightarrow[T \to +\infty]{(d)} X^\natural $ & Convergence rate of $\Lambda^{T,\natural}$ to $X^\natural $  \\ \hline
        NUHP \newline$(\natural = -)$ & Theorem 2.2, \cite{jaisson_limit_2015} & \textit{\underline{Theorem \ref{th: NUHP to X}}} \\ \hline
         WCHP \newline$(\natural = 0)$ & Theorem 2.10, \cite{horst_functional_2024} & $\bullet$ Convergence rate in 1-Wasserstein distance: Theorem 2.12, \cite{horst_functional_2024} \newline $\bullet$ \underline{\textit{Convergence rate in $\LL^2\left( \Omega, \LL^\infty([0,1])\right)$:}} \underline{\textit{Theorem \ref{th: NUHP to X} }}\\ \hline
         SNUHP \newline$(\natural = +)$ & $\bullet$ Convergence of the rescaled counting process: Theorem 1.2, \cite{liu_scaling_2024} \newline $\bullet$ \textit{\underline{Convergence of the intensity:}}\newline \textit{\underline{Theorem \ref{th: NUHP to X}}} & \textit{\underline{Theorem \ref{th: NUHP to X}}} \\ \hline
    \end{tabular}
    \caption{State of the art of convergence results regarding the rescaled intensity}
    \label{tab:int_conv_result}
\end{table}

~\\[2ex]
~\\
~\\
\begin{center}
    \textit{How can we provide a unified framework for these critical limit theorems?}
\end{center}
We construct a two-parameter coupling between a compensated Poisson random measure on $[0,1]\times\RR_+$ and a Brownian sheet $W$ on the same space that allows us to provide general bounds for the convergence of stochastic integrals with respect to these measures (see Theorem \ref{thm: maj_sup_int_NTW}). The coupling we produce is strong as it allows to derive functional estimates in the uniform topology (and thus in 2-Wasserstein distance). This unified framework allows us to capture in the three cases NUHP, WCHP and SNUHP.\\
More precisely, a central object in the analysis of Hawkes processes is the function $\Psi^T$, defined as the solution to the Volterra equation associated with the kernel $\phi^T$, that is,
$$\Psi^T = \phi^T + \phi^T * \Psi^T.$$
This function plays a crucial role in rewriting the intensity in a form more suitable for our analysis:
\begin{equation}
\label{eq:Lambda_intro}
\Lambda^{T,\natural}_t = \frac{\mu}{T} + \mu \int_0^t \Psi^{(T),\natural}(t-s) \, ds + \iint_{(0,t)\times \RR_+} \Psi^{(T),\natural}(t-s) \, \1_{\theta \leq \Lambda^{T,\natural}_s} \, \Ntil^T(ds,d\theta), \quad t \in [0,1],
\end{equation}
where $\natural \in \{-, 0, +\}$, $\Psi^{(T),\natural} := \Psi^{T,\natural}(T \cdot)$, and $\Ntil^T$ is a rescaled, compensated two-parameter Poisson random measure defined in Subsection \ref{subs: hawkes proc}.
It is known in the literature that, for each $\natural \in \{-, 0, +\}$, the rescaled function $\Psi^{(T),\natural}$ converges (in an appropriate sense) to a limiting function $\rho^\natural$. This convergence is one of the key components we quantify to establish our convergence rates.
Moreover, each limiting process $X^\natural$ admits a representation involving a Brownian sheet $W$:
\begin{equation}
\label{def: CIR}
X^\natural_t = \mu \int_0^t \rho^\natural(t-s) \, ds + \iint_{(0,t)\times \RR_+} \rho^\natural(t-s) \, \1_{\theta \leq X^\natural_s} \, W(ds,d\theta), \quad t\in [0,1].
\end{equation}
~\\

We proceed as follows. Notations and Hawkes processes properties are presented in Section \ref{sec: notation et preliminaire}. The coupling and its asymptotic properties are introduced in Section \ref{sec: coupling}.  The upper bound for the convergence of the different processes to their limit is presented in Section \ref{sec: main upper bound}. Finally technical lemmata, proof of Propositions \ref{prop: conv_psi_rho_-}, \ref{prop: conv_psi_rho_0}, \ref{prop: conv_psi_rho_+}  and convergence results regarding the kernels are postponed to Section \ref{sec: lemmata}.

\section{Notations and preliminaries}
\label{sec: notation et preliminaire}

\subsection{Notations}
 We denote by $\NN$ (resp. $\NN^*$) the set of non-negative (resp. positive) integers, that is, $\NN := \{0, 1, 2, \dots\}$ (resp. $\NN^* := \{1, 2, \dots\}$). Similarly, we define the sets of non-negative and positive real numbers as $\RR_+ := [0, +\infty)$ and $\RR_+^\ast := (0, +\infty)$, respectively.  

For any $(a,b) \in \RR^2$, we adopt the convention of writing $[a,b]=[\min(a,b), \max(a,b)]$ regardless of the order between $a$ and $b$. Moreover, we denote by $\Bcal(E)$ the Borel $\sigma$-algebra of a given topological space $E$.  

For any $(i,k) \in \NN \times \NN^*$, we introduce the notations  
$$ I_{i,k} := \left(t^k_i,t^k_{i+1} \right], \quad A_{i} := \left(0,t_i^k \right]^2, \quad \text{where} \quad t^k_i := \frac{i}{k}. $$  

Regarding function spaces, for any $A \in \Bcal(\RR)$ and $p \in \NN^*$, we define the Lebesgue space  
$$ \LL^p(A) = \left\{ f:A \to \RR \mid \|f\|_{\LL^p(A)} := \left(\int_A |f(t)|^p dt \right)^{1/p} < +\infty \right\}.$$  
For $p = \infty$, we set the set of a.e. bounded functions
$$ \LL^{\infty}(A) = \left\{ f:A \to \RR \mid \|f\|_{\LL^{\infty}(A)} := \esssup_{t\in A} |f(t)| < +\infty \right\}. $$  
In the specific case $A = \RR_+$, we use the shorthand notation $\LL^p := \LL^p(\RR_+)$ for any $p \in \NN^* \cup \{\infty\}$. \\

Throughout this paper, unless otherwise specified, we assume that $T \geq 2$. We also use $C>0$ to denote a constant that may change from line to line but remains independent of $T$, which is the key parameter we seek to control in our results.  

\subsection{Hawkes processes}
\label{subs: hawkes proc}
We start this section by providing the definition of a linear Hawkes process $H$ with parameter $\mu>0$ and $\phi: \RR_+ \to \RR_+$.
\begin{definition}[Hawkes process, \cite{bremaud_stability_1996}]
\label{def:Hawkes}
Let $(\Omega, \GG := \left( \Gcal_t\right)_{t\geq 0}, \PP)$ be a filtered probability space.
Let $\mu>0$ and $\phi:\RR_+ \to \RR_+$. A linear Hawkes process $H:=(H_t)_{t \geq 0}$ with parameters $\mu$ and $\phi$ is a counting process such that   
\begin{itemize}
\item[(i)] $H_0=0,\quad \PP-a.s.$,
\item[(ii)] its ($\GG$-predictable) intensity process is given by
$$\lambda_t:=\mu + \int_{(0,t)} \phi(t-s) dH_s, \quad t\geq 0,$$
that is for any $0 \leq s \leq t $ and $A \in \Gcal_s$,
$$ \E{\1_A (H_t-H_s)} = \E{\int_{(s,t]} \1_A \lambda_r dr }.$$
\end{itemize}
\end{definition}

The definition used in this paper differs from the standard one. Instead, we adopt an alternative approach by representing the Hawkes process through the solution of a stochastic differential equation (SDE). This representation, introduced in \cite{bremaud_stability_1996}, relies on a two-parameter Poisson measure $N$, as described below.

We begin by defining the space of configurations:
$$ \Omega:=\left\{\omega=\sum_{i=1}^{n} \delta_{(t_{i},\theta_i)} \mid 0\leq t_0 < t_1 < \cdots < t_n, \; \theta_i \in \RR_+, \; n\in \NN \cup\{+\infty\} \right\}.$$
Each path of a counting process is represented as an element $\omega$ in $\Omega$ which is a $\NN$-valued measure on $\RR_+^2$. 

Let $\Fcal^N$ be the $\sigma$-field associated to the vague topology on $\Omega$, and $\PP$ be a probability measure under which the random measure $N$ defined as:
$$ N(B)(\omega):=\omega(B), \quad B \in \Bcal \left(\RR_+^2\right),$$
is a random measure with intensity $1$ (so that $N(B)$ is a Poisson random variable with intensity $\pi(B)$ for any $B\in \Bcal \left(\RR_+^2 \right)$ and where $\pi$ denotes the Lebesgue measure that is $\pi(B)=\iint_{\RR_+^2} \1_B(u,\theta) du d\theta $). We set $\FF^N:=(\Fcal_t^N)_{t\in \RR_+}$ the natural history of $N$, that is $$\Fcal_t^N:=\sigma \left(N( \Tcal  \times B), \; \Tcal \subset \Bcal((0,t]), \; B \in \Bcal(\RR_+) \right).$$ Let also, $\Fcal_\infty^N:=\lim_{t\to+\infty}\Fcal_t^N$ .\\
With this notations on hand, we can define a Hawkes process through an SDE that we present in Theorem \ref{th:HRR} and Corollary \ref{cor: poisson_imbedding} both proved in \cite{coutin_normal_2024}. However, we do not write it exactly as presented in \cite{coutin_normal_2024} since an analysis of the proof gives that $\phi$ positive and locally integrable is a sufficient assumption to get their result
\begin{theorem}
    \label{th:HRR}
    Suppose that $\phi$ is positive and locally integrable. Then, the SDE below admits a unique $\FF^N$-predictable solution $\lambda$: 
    \begin{equation}
    \label{eq:Int_Hawkes}
    \lambda_t = \mu + \iint_{(0,t)\times \RR_+} \phi(t-u)\ind{\theta \leq \lambda_u} N(du,d\theta) ,\quad t \in \RR_+ .
    \end{equation}
\end{theorem}
\begin{cor}
    \label{cor: poisson_imbedding}
    Suppose that the assumptions of Theorem \ref{th:HRR} hold. We denote by $\lambda$ the unique solution of \eqref{eq:Int_Hawkes} and by $H$ his associated counting process, i.e.  
    $$H_t := \iint_{(0,t]\times \RR_+} \1_{\{\theta \leq \lambda_s\}} N(ds,d\theta), \quad t\in \RR_+.$$
    Then, $H$ is a Hawkes process in the sense of Definition \ref{def:Hawkes} whose intensity is $\lambda$.
\end{cor}

In this paper, the norm of the kernel is supposed to converge to $1$ with a parameter $T$. To do so we introduce a parameter $a_T$ such that $\lim_{T\to +\infty }a_T =1$. We also introduce a sequence of Hawkes processes indexed by $T$, denoted $\left(H^T_t \right)_{t\geq 0}$ with the following intensity:
\begin{equation}
    \label{eq: int_nuhp}
    \lambda^T_t = \mu + \int_{(0,t)} \phi^T(t-s) dH^T_s, \quad t\in [0,T]
\end{equation}
where $\mu>0$ and $\phi^T$ is a nonnegative measurable function on $\RR_+$ which satisfies $\left\| \phi^T \right\|_1 \to 1$.

\begin{assumption}
\label{assump: phi}
    For $t\in \RR_+$, $\phi^T(t) = a_T \phi(t)$ with $\lim\limits_{T\to +\infty}a_T = 1$ and $\phi$ is a nonnegative measurable function such that 
    $$\left\| \phi \right\|_1 = 1, \quad m:= \int_0^{\infty} s\phi(s) ds < +\infty, \quad \text{and}\quad m_2:= \int_0^{\infty} s^2\phi(s) ds < +\infty.$$
    Moreover, we suppose that $\phi$ is differentiable with derivative $\phi'$ such that $\left\| \phi'\right\|_{\infty}<+\infty$ and $\left\| \phi'\right\|_{1}<+\infty$.
\end{assumption}

\begin{remark}
    Note that under Assumption \ref{assump: phi}, we also have that $\phi \in \LL^{\infty}(\RR_+)$ which yields $\phi \in \LL^2(\RR_+)$.
\end{remark}

We also set the functions $\Psi^T$ and $\Psi^{(T)}$, both defined on $\RR_+$, by 
\begin{equation}
\label{eq:def_psiT}
    \Psi^T := \sum_{k\geq 1} \left(\phi^T\right)^{\ast k} = \sum_{k\geq 1} (a_T)^k \left(\phi\right)^{\ast k} \quad \text{and} \quad \Psi^{(T)}= \Psi^T(T \cdot)
\end{equation}
where $\left(\phi^T\right)^{\ast 1} = \phi^T$ and for any $k\geq 2$, $\left(\phi^T\right)^{\ast (k-1)} = \left(\phi^T\right)^{\ast k} \ast \phi^T$. It is important to note that existence and uniqueness of $\Psi^T$ are due to the fact that $\phi$ is locally integrable (see Theoreom 1.1 of \cite{grossman_notes_1979} for instance). We are now able to write an SDE on the intensity using the two-parameter Poisson measure and $\Psi^T$.
\begin{prop}[Proposition 2.1 in \cite{jaisson_limit_2015}]
    For all $t\geq 0$, we have
    $$\lambda^T_t = \mu + \mu\int_0^t \Psi^T(t-s) ds + \iint_{(0,t)\times \RR_+} \Psi^T(t-s) \1_{\theta \leq \lambda^T_s} \Ntil(ds,d\theta).$$
\end{prop}

In the following, we slightly modify the measure in the SDE solved by $\lambda^T$. Indeed, in \cite{jaisson_limit_2015} and \cite{liu_scaling_2024}, the convergence relies on the convergence of a rescaled compensated Poisson measure to a Gaussian measure. Thus, we define another two-parameter random measure called $N^T$. For any $(a,b,c,d)\in \RR_+^4$, we define the signed (non-integer valued) random measure $N^T$ and its compensated version $\Ntil^T$ as follows:
\begin{align*}
    N^T\left( [a,b]\times [c,d]\right) &= \frac{1}{T} N\left( [aT,bT]\times [cT,dT]\right);\\
    \Ntil^T\left( [a,b]\times [c,d]\right) &=N^T\left( [a,b]\times [c,d]\right) -T(b-a)(d-c).
\end{align*}
We make use of the following notations $\FF^T:=(\Fcal^N_{tT})_{t\geq 0}$ and $N^T(s,t) := N^T\left( [0,s]\times [0,t]\right)$ for any $(s,t)\in \RR_+^2$.
\begin{lem}\label{lem-outil-tilde-N-T}
\begin{itemize}
    \item[(i)] Let $(u_{(s,\theta)})_{(s,\theta)\in \RR_+^2}$ be a $\FF^T$ predictable process 
    such that
    \begin{align*}
        \E{ \int_0^1 \int_{\RR_+} u_{(s,\theta)}^2 d\theta ds + \left|\int_0^1 \int_{\RR_+} |u_{(s,\theta)}| d\theta ds \right|^2 } <+\infty.
    \end{align*}
    Then, the process $(\iint_{(0,t] \times \RR_+} u_{(s,\theta)}\Ntil^T(ds,d\theta) )_{t\in [0,1]}$ is a right-continuous square integrable martingale with bracket
    $$ \left[\iint_{(0,\cdot] \times \RR_+}u_{(s,\theta)}\Ntil^T(ds,d\theta)\right]_t=\frac{1}{T}\iint_{(0,t] \times \RR_+}u_{(s,\theta)}^2 N^T(ds,d\theta), \quad t\in [0,1], .$$
    \item[(ii)] For $u$ fulfilling, in addition of assumption of (i),
   $$ \E{  \int_0^1 \int_{\RR_+} u_{(s,\theta)}^4 d\theta ds + \left(\int_0^1 \int_{\RR_+}  u_{(s,\theta)}^2 d\theta ds\right) ^{2} }<+\infty,$$
   there exists a positive constant $C$ such that
   \begin{align*}
       \hspace*{-1.5cm}\E{\sup_{t \in [0,1]}\left[ \iint_{(0,t]\times \RR_+}u_{(s,\theta)} \Ntil^T(ds,d\theta)\right] ^{4}} \leq& C\E{ \left|\int_0^1 \int_{\RR_+} u_{(s,\theta)}^2 d\theta ds \right|^2 }+ \frac{C}{T^2}\E{\int_0^1 \int_{\RR_+} u_{(s,\theta)}^4 d\theta ds}.
   \end{align*}
\end{itemize}
\end{lem}
\begin{proof}
    The proof relies on the fact that $TN^T$ is a Poisson measure with intensity $T^2 dsd\theta.$
\end{proof}
 
Hence, if we define by $\Lambda^T$ the time-and-space rescaled intensity of Hawkes processes, we can use $N^T$ and write $\Lambda^T$ as the solution of this SDE:
\begin{equation}
\label{def: Lambda^T}
    \Lambda^T_t := \frac{\lambda^T_{tT}}{T} = \frac{\mu}{T} + \mu\int_0^t \Psi^{(T)}(t-s) ds + \iint_{(0,t)\times \RR_+} \Psi^{(T)}(t-s) \1_{\theta \leq \Lambda^T_s} \Ntil^T(ds,d\theta), \quad t\in [0,1].
\end{equation}
where $\Psi^{(T)}$ is defined in \eqref{eq:def_psiT}. Note that $\Lambda^T$ is a càg-làd semi-martingale because of Assumption \ref{assump: phi}. Indeed, since $\phi$ (and so $\Psi^{(T)}$) is differentiable, we can use Taylor expansion formula and Fubini's theorem to get:
\begin{align*}
    \Lambda^T_t &= \frac{\mu}{T} + \mu\int_0^t \Psi^{(T)}(s) ds + \int_0^t \left(\Psi^{(T)}\right)' (u)\left[\iint_{(0,u)\times \RR_+}  \1_{\theta \leq \Lambda^T_s} \Ntil^T(ds,d\theta) \right]du \\
    &+ \Psi^{(T)}(0)\iint_{(0,t)\times \RR_+} \1_{\theta \leq \Lambda^T_s} \Ntil^T(ds,d\theta).
\end{align*}

We will now discuss different Hawkes processes regarding the sign of $1-a_T$. We then introduce the following notations:
\begin{notation}
\label{not: diff_cas}
    \begin{itemize}
        \item When $a_T < 1$, we make use of "$-$" as an exponent for all processes that are linked to this case.
        \item When $a_T > 1$, we make use of "$+$" as an exponent for all processes that are linked to this case.
        \item When $a_T = 1$, we may use of "$0$" as an exponent for all processes that are linked to this case.
    \end{itemize}
    We summarize these notations in the Table \ref{tab: notation}:  
    \renewcommand{\arraystretch}{1.5}
    \begin{table}[!ht]
        \centering
        \begin{tabular}{|c|c|c|c|c|c|c|c|}
            \cline{2-8} \multicolumn{1}{c|}{} & $a_T$ & $\phi^T$ & $\phi^{(T)}$ & $\Psi^{T}$ & $\Psi^{(T)}$ & $\Lambda^{T}$ & $H^{T}$\\  \hline 
             NUHP $(\natural = -) $& $a^{-}_T<1$ & $\phi^{T,-}$ & $\phi^{(T),-}$ & $\Psi^{T,-}$ & $\Psi^{(T),-}$ & $\Lambda^{T,-}$ & $H^{T,-}$ \\ \hline 
            WCHP $(\natural = 0) $& $a^0_T=1$ & $\phi$ & $\phi^{(T),0}$ & $\Psi$ & $\Psi^{(T),0}$ & $\Lambda^{T,0}$ & $H^{T,0}$ \\ \hline 
             SNUHP $(\natural = +) $ & $a^{+}_T>1$ & $\phi^{T,+}$ & $\phi^{(T),+}$ & $\Psi^{T,+}$ & $\Psi^{(T),+}$ & $\Lambda^{T,+}$ & $H^{T,+}$ \\  \hline
        \end{tabular}
        \caption{Notations}
        \label{tab: notation}
    \end{table}
\end{notation}

Note that we will employ these different notations when required by the context. However, for the sake of simplicity, we will primarily use the general notations introduced earlier, such as $a^{\natural}_T$, $\Lambda^{T,\natural}$, $\Psi^{T,\natural}$, etc., with $\natural \in \{ -,0,+\}$, when the result can be applied to any Hawkes processes (NUHP, WCHP, or SNUHP).

\paragraph{ \textbf{\texorpdfstring{$a_T <1$ : Nearly Unstable Hawkes Process (NUHP)}{} }}~\\
The first case that we discuss in this paper is about the NUHP. We define such process by taking $a_T<1$. There were originally presented in \cite{jaisson_limit_2015}. In particular, the authors discuss the convergence rate of $(a_T)_{T\geq 2}$ to $1$ and explain the importance to have $\lim\limits_{T\to +\infty}T(1-a_T)=1$. They explain that the behavior of a rescaled Hawkes process is nondegenerate (neither explosive, nor deterministic) with such $(a_T)_{T\geq 2}$. Therefore we now suppose that
$$a_T^{-} = 1- \frac{1}{T}.$$
With such assumption, \cite{jaisson_limit_2015} proved the following convergence:
\begin{theorem}[Theorem 2.2 in \cite{jaisson_limit_2015}]
    \label{thm: JR}
    Under Assumption \ref{assump: phi} and assuming that $\sup_{T>1}\left\| \Psi^{T} \right\|_{\infty} <+\infty$, the sequence of renormalized Hawkes intensities $\left( \Lambda^{T,-}_t \right)_{t\in [0,1]}$ defined in \eqref{def: Lambda^T} converges in law, for the Skorohod topology, toward the law of the unique strong solution of the following Cox–Ingersoll–Ross SDE on $[0, 1]$:
    $$X^{-}_t = \frac{1}{m}\int_0^t (\mu - X^{-}_s) ds + \frac{1}{m}\int_0^t \sqrt{X^{-}_s} dB_s$$
    where $B$ stands for the standard Brownian motion and $m= \int_0^{\infty} s \phi(s) ds$.
\end{theorem}
To prove the previous result, the authors use the $\LL^2(0,1)$ convergence of $\Psi^{(T),-}$ to an exponential (see Lemma 4.5 in \cite{jaisson_limit_2015}). In this paper, we decided to denote by $\rho^{-}$ the limit function, that is:
\begin{equation}
    \rho^{-}: x\in \RR \mapsto \frac{1}{m}e^{-x/m}, \quad m= \int_0^\infty t \phi(t) dt.
\end{equation}
The convergence rate of this result is one of the key point of our work and it is given in Proposition \ref{prop: conv_psi_rho_-}. Nevertheless we postpone the proof in Subsection \ref{subs: proof conv_psi_rho_-}.
\begin{prop}
    \label{prop: conv_psi_rho_-}
    There exists $C>0$ such that for $T>2$, $\left\|\Psi^{(T),-}-\rho^{-} \right\|_{\LL^2(0,1)} \leq \frac{C}{\sqrt{T}}$.
\end{prop}

These results were obtained under some strong assumption on the kernel. Whereas we will place our work under Assumption \ref{assump: phi}, it was proved in \cite{horst_functional_2024} that it is not necessary to get this to obtain the convergence. Indeed, the authors proved similar results under minimal conditions on the kernel of the process by assuming that the kernel is a  slowly varying function.

\paragraph{ \textbf{\texorpdfstring{$a_T =1$: Weakly Critical Hawkes Process (WCHP)}{} }}~\\
We now focus on the critical case that is the case $a_T = 1$. This case was also studied in \cite{horst_functional_2024}. In particular, they obtain the following results:
\begin{theorem}[Theorem 2.10 in \cite{horst_functional_2024}]
    If $\left\|\phi \right\|_{1} =1 $ and $m = \int_0^{\infty } s \phi(s) ds < +\infty$,
    the sequence of renormalized Hawkes intensities $\left( \Lambda^{T,0}_t \right)_{t\in [0,1]}$ defined in \eqref{def: Lambda^T} converges in law, for the Skorohod topology, toward the law of the unique strong solution of the following SDE on $[0, 1]$:
    $$X^{0}_t = \frac{\mu t }{m} + \frac{1}{m}\int_0^t \sqrt{X^{0}_s} dB_s$$
    where $B$ stands for the standard Brownian motion.
\end{theorem}
Their convergence result is also presented with the convergence rate using Wasserstein distance which we partially recall in Theorem \ref{thm: horst_vitesse_convergence}. Note that \cite{horst_functional_2024} gives the first convergence rate established for the behavior of WCHP.
\begin{theorem}[Theorem 2.12, \cite{horst_functional_2024}]
\label{thm: horst_vitesse_convergence}
    If $\phi$ has bounded variation, then for each $\kappa \in (0, 1/2)$ and each $K>0$, there exists a constant $C>0$ such that for any $T \in \NN^\ast$, we have
    \begin{equation*}
    \label{eq: borne_Horst_Xu}
        d_{\text{\cite{horst_functional_2024}}}\left(\Lambda^{T,0}, X^0\right) := \sup_{f\in \mathrsfs{B}_K }\sup_{t\in [0,1]} \Wcal_1 \left( f * \Lambda^{T,0} (t) , f * X^0 (t) \right) \leq C \left| \ln \left( \frac{1}{T} + \left\| \Psi^{(T),0}-\rho^0 \right\|_{\LL^2 (0,1)} \right)\right|^{-\kappa},
    \end{equation*}
    where $\mathrsfs{B}_K := \left\{ f\in \LL^{\infty}(\RR_+, \CC) \mid \left\| f\right\|_{\infty}\leq K \right\}$, $\rho^0 \equiv \frac{1}{m}$ and for any $f\in \mathrsfs{B}_K$,
    $$f * \Lambda^{T,0} (t) := \int_0^t f(t-s) d\Lambda^{T,0}_s \quad \text{and} \quad f * X^{0} (t) := \int_0^t f(t-s) dX^{0}_s, \quad t\in [0,1].$$
\end{theorem}
This theorem gives rise to the following remark.
\begin{remark}
\label{rem: horst_xu}
   While the estimate of \cite{horst_functional_2024} concerns the 1-Wasserstein distance, we instead work with the $\LL^2\left(\Omega, \LL^\infty(0,1) \right)$ norm, which dominates the 2-Wasserstein distance and which allows us indeed to control $\E{\sup_{t\in [0,1]} \left| \Lambda^{T,0}_t - X^0_t\right|}$ as
    $$\sup_{t\in [0,1]}\Wcal_1(\Lambda^{T,0}_t, X^0_t) \leq \sup_{t\in [0,1]}\Wcal_2(\Lambda^{T,0}_t, X^0_t) \leq \E{\sup_{t\in [0,1]} \left| \Lambda^{T,0}_t - X^0_t\right|} \leq \sqrt{\E{\sup_{t\in [0,1]} \left| \Lambda^{T,0}_t - X^0_t\right|^2}}.$$
    It is not straightforward to us that norm $d_{\text{\cite{horst_functional_2024}}}$ is dominated by $\E{\sup_{t}\left| \cdot \right|^p},\, p=1,2$, a limitation that we overcome in our work. We extend their approach by considering an alternative distance and incorporating the supremum norm, which yields an upper bound in the Skorokhod topology.\\
    Compared to the work of \cite{horst_functional_2024}, our result is established under stronger assumptions, due to our choice to work in the $\LL^2\left(\Omega, \LL^\infty(0,1) \right)$ norm, which naturally requires more regularity on the kernel.
\end{remark}

A key element in Theorem \ref{thm: horst_vitesse_convergence} is the $\LL^2(0,1)$ convergence of the resolvent $\Psi^{(T),0}$ towards the constant function $\rho^0 \equiv \frac{1}{m}$. More precisely, they derive an upper bound in $\LL^2$ norm for the integrated version of $\Psi^{(T),0}$ and $\frac{1}{m}$. While Proposition 2.13 in \cite{horst_functional_2024} is originally formulated in terms of integrals, a closer analysis of its proof leads to the following refined result:
\begin{prop}[Proposition 2.13 in \cite{horst_functional_2024}]
\label{prop: conv_psi_rho_0}
    There exists $C>0$ such that for any $T>2$, $\left\|\Psi^{(T),0}-\rho^{0} \right\|_{\LL^2(0,1)} \leq \frac{C}{\sqrt{T}}$ where $\rho^0$ stands for the constant function $\rho^0(t)=\frac{1}{m}$.
\end{prop}

\paragraph{ \textbf{\texorpdfstring{$a_T >1$ : Supercritical Nearly Unstable Hawkes Process (SNUHP)}{} }}~\\
The last case that we discuss in this paper is $a_T>1$. Inspired by the work of \cite{jaisson_limit_2015}, the authors of \cite{liu_scaling_2024} obtain a convergence result for the SNUHP. In particular, they get the following theorem:
\begin{theorem}[Theorem 1.2, \cite{liu_scaling_2024}]
    \label{thm: liu}
    Under Assumptions \ref{assump: phi} and assuming that $\sup_{T>1}\left\| \Psi^{T} \right\|_{\infty} <+\infty$, 
    the sequence of renormalized Hawkes process $\left(H^{T,+}_{tT}/T^2 \right)_{t\in [0,1]}$ defined in \eqref{def: Lambda^T} converges in law, for the Skorohod topology, towards the law of $\left( \int_0^t X^{+}_s ds\right)_{t\in [0,1]}$ where 
    $$X^{+}_t = \frac{1}{m}\int_0^t (\mu + X^{+}_s) ds + \frac{1}{m}\int_0^t \sqrt{X^{+}_s} dB_s$$
    with $B$ standing for the standard Brownian motion and $m= \int_0^{\infty} s \phi(s) ds$.
\end{theorem}
As for the previous cases, a key element of this convergence come from the convergence of the resolvent:
\begin{prop}
    \label{prop: conv_psi_rho_+}
    There exists $C>0$ such that for $T>2$, $\left\|\Psi^{(T),+}-\rho^{+} \right\|_{\LL^2(0,1)} \leq \frac{C}{\sqrt{T}}$ where $\rho^{+}$ is the function defined on $\RR_+$ by
    $$\rho^{+}(x) := \frac{1}{m} e^{x/m}, \quad m= \int_0^\infty t \phi(t) dt.$$
\end{prop}
We now make a few remarks concerning their theorem and the convergence of $\Psi^{(T),+}$.
\begin{remark}
Note that Theorem \ref{thm: liu} does not establish the convergence of the intensity process, but only that of the counting process. This gap in the literature will be addressed in Theorem \ref{th: NUHP to X}, where we prove the convergence of the intensity. Moreover, a quantitative version of their result will follow directly from our analysis, as detailed in Corollary \ref{cor: autre maj}.\\
In addition, in \cite{liu_scaling_2024}, the convergence of $\Psi^{(T),+}$ is only shown pointwise (see Lemma 2.2 in \cite{liu_scaling_2024}). In contrast, we refine their analysis by establishing convergence in $\LL^2(0,1)$.
\end{remark}

\section{\texorpdfstring{Coupling between $\Ntil^T $ and the Brownian sheet $W$}{}}
\label{sec: coupling}

The second main convergence in this work is leaded by the Poisson imbedding structure. More precisely, we analyze the convergence of the rescaled compensated Poisson measure $\Ntil^T$ towards the Brownian sheet $W$, in a sense that is made precise in this section. This coupling forms the backbone of our quantitative analysis and its construction is given in this section. Since the construction of $W$ relies on a discretisation $(t_i^k)_{i\in \NN}$, we need to keep in mind that $W$ depends on $k\in \NN^\ast$ and $T$. However, we decide for simplicity not to indicate this dependence in the notation of $W$ since our constructed Brownian sheets are all equal in law to the standard Brownian sheet. We now fix the parameter $k\in \NN^\ast$ and we will give it a specific value later.

\subsection{Construction of \texorpdfstring{$W$}{}}
In \cite{fomichov_implementable_2021}, the authors present an algorithm allowing the coupling between a Brownian motion and a Lévy process. In this section, we have adapted their proof to create a similar algorithm and to couple $\Ntil^T $ with a Brownian sheet. We start this subsection with Proposition \ref{prop: construction grand drap}. This proposition will be the main tool of the construction of our Brownian sheet that is presented in Theorem \ref{thm: Couplage}.

\begin{prop}
\label{prop: construction grand drap}
    Fix $k\geq 1$. Let $B_{i,j}$ be a family of independent Brownian sheet on $\left[0,1/k\right]^2$. The process $W^B$ defined by $W^B\left(0,t \right)=W^B\left(s,0\right)=0$ and 
    \begin{align*}
   W^B\left(s,t \right) &= W^B\left( s, t^k_{j-1} \right)+W^B\left( t^k_{i-1}, t \right)-W^B\left( t^k_{i-1}, t^k_{j-1} \right) + B_{i,j}\left(s-t^k_{i-1},t-t^k_{j-1}\right)
   \end{align*}
   for $t_{i-1}^k < s \leq t_{i}^k$, $t_{j-1}^k < t \leq t_{j}^k = \frac{j}{k}$ and $(i,j) \in \NN^\ast \times \NN^\ast $
   is a Brownian sheet on $\RR_+^2$ that is a two-parameter centered Gaussian process vanishing on the axes and whose convariance is given by
   $$\E{W^B(s_1,t_1)W^B(s_2,t_2)} = (s_1 \wedge s_2)(t_1 \wedge t_2), \quad \forall (s_1,s_2,t_1,t_2) \in \RR_+^4 .$$
\end{prop}
\begin{proof}
    We prove the result using induction. To do so, we employ for any $\ell\in \NN^\ast$ the notation $A_{\ell} := \left( 0, t^k_{\ell}\right]^2.$\\
    For $\ell=1$, we have $W^B(s,t) = B_{1,1}\left(s,t \right)$ for any $(s,t) \in A_{1}$. Then, $W^B$ is a Brownian sheet on $A_{1}$.\\
    Let $\ell \geq 2$. Suppose that $W^B$ is a Brownian sheet on $A_{\ell-1}$. For any $(s,t), (s',t') \in A_{\ell}$, we have:
        \begin{align*}
            \E{W^B(s',t')W^B(s,t)} =& \E{ W^B(s',t') B_{i,j}\left(s-t^k_{i-1},t-t^k_{j-1} \right)}\\
            &+ \E{W^B(s',t') \left(W^B\left(s, t^k_{j-1} \right)+W^B\left(t^k_{i-1},t \right)-W^B\left(t^k_{i-1}, t^k_{j-1} \right) \right)}
        \end{align*}
        where $(s,t)\in I_{i,k}\times I_{j,k}$ and $\max(i,j) \leq \ell$. We set some notations and split the proof in different cases explained in Figure \ref{fig: coupling}.
\begin{figure}[ht]
\subfigure[Notations \label{fig: notations}]{%
\begin{tikzpicture}[scale=0.7]
    \draw[->, very thick] (0,0)--(0,5);
    \draw[->, very thick] (0,0)--(5,0);
    \draw[black,fill=gray!20] (0,0) rectangle (3,3);
    \draw (0,0) rectangle (4.5,4.5);
    \draw[dashed] (3,0)--(3,4.5);
    \draw[dashed] (0,3)--(4.5,3);
    \draw (3,0) node[below] {$t_{\ell-1}^k$};
    \draw (4.5,0) node[below] {$t_{\ell}^k$};
    \draw (0,3) node[left] {$t_{\ell-1}^k$};
    \draw (0,4.5) node[left] {$t_{\ell}^k$};
    \draw (1.6,1.6) node {$A_{\ell-1}$};
    \draw (3.8,3.8) node {$I_{\ell-1,k}^2$};
    \draw (1.6,3.8) node {$\Jtil_{\ell-1}$};
    \draw (3.8,1.6) node {$J_{\ell-1}$};
    \draw (0.5,0.5) node[right] {$x=(s,t)$} node{$\bullet$};
\end{tikzpicture}}
\hfill
\subfigure[ $(x,x')\in A_{\ell-1}^2$\label{fig: Cas 1}]{%
\begin{tikzpicture}[scale=0.7]
    \draw[->, very thick] (0,0)--(0,5);
    \draw[->, very thick] (0,0)--(5,0);
    \draw[black,fill=gray!20] (0,0) rectangle (3,3);
    \draw (0,0) rectangle (4.5,4.5);
    \draw[dashed] (3,0)--(3,4.5);
    \draw[dashed] (0,3)--(4.5,3);
    \draw (3,0) node[below] {$t_{\ell-1}^k$};
    \draw (4.5,0) node[below] {$t_{\ell}^k$};
    \draw (0,3) node[left] {$t_{\ell-1}^k$};
    \draw (0,4.5) node[left] {$t_{\ell}^k$};
    \draw (0.5,2.3) node[below right=-0.05cm] {$x$} node{$\bullet$};
    \draw (1.3,0.9) node[below right=-0.1cm] {$x'$} node{$\bullet$};
\end{tikzpicture}}
\hfill
\subfigure[ $x'\in A_{\ell-1}$ and $x\in J_{\ell-1}$ \label{fig: Cas 2}]{%
     \begin{tikzpicture}[scale=0.7]
    \draw[->, very thick] (0,0)--(0,5);
    \draw[->, very thick] (0,0)--(5,0);
    \draw[black,fill=gray!20] (0,0) rectangle (3,3);
    \draw (0,0) rectangle (4.5,4.5);
    \draw[dashed] (3,0)--(3,4.5);
    \draw[dashed] (0,3)--(4.5,3);
    \draw (3,0) node[below] {$t_{\ell-1}^k$};
    \draw (4.5,0) node[below] {$t_{\ell}^k$};
    \draw (0,3) node[left] {$t_{\ell-1}^k$};
    \draw (0,4.5) node[left] {$t_{\ell}^k$};
    \draw (3.3,2) node[below right=-0.05cm] {$x$} node{$\bullet$};
    \draw (1.3,0.9) node[below right=-0.1cm] {$x'$} node{$\bullet$};
\end{tikzpicture}
}
\subfigure[ $x\in A_{\ell-1}$ and $x\in I_{\ell-1,k}^2$\label{fig: Cas 3}]{%
\begin{tikzpicture}[scale=0.7]
    \draw[->, very thick] (0,0)--(0,5);
    \draw[->, very thick] (0,0)--(5,0);
    \draw[black,fill=gray!20] (0,0) rectangle (3,3);
    \draw (0,0) rectangle (4.5,4.5);
    \draw[dashed] (3,0)--(3,4.5);
    \draw[dashed] (0,3)--(4.5,3);
    \draw (3,0) node[below] {$t_{\ell-1}^k$};
    \draw (4.5,0) node[below] {$t_{\ell}^k$};
    \draw (0,3) node[left] {$t_{\ell-1}^k$};
    \draw (0,4.5) node[left] {$t_{\ell}^k$};
    \draw (3.3,3.7) node[below right=-0.05cm] {$x$} node{$\bullet$};
    \draw (1.3,0.9) node[below right=-0.1cm] {$x'$} node{$\bullet$};
\end{tikzpicture}}
\hfill
\subfigure[ $(x,x')\in J_{\ell-1}\times \Jtil_{\ell-1}$\label{fig: Cas 4}]{%
\begin{tikzpicture}[scale=0.7]
    \draw[->, very thick] (0,0)--(0,5);
    \draw[->, very thick] (0,0)--(5,0);
    \draw[black,fill=gray!20] (0,0) rectangle (3,3);
    \draw (0,0) rectangle (4.5,4.5);
    \draw[dashed] (3,0)--(3,4.5);
    \draw[dashed] (0,3)--(4.5,3);
    \draw (3,0) node[below] {$t_{\ell-1}^k$};
    \draw (4.5,0) node[below] {$t_{\ell}^k$};
    \draw (0,3) node[left] {$t_{\ell-1}^k$};
    \draw (0,4.5) node[left] {$t_{\ell}^k$};
    \draw (1.4,3.7) node[below right=-0.1cm] {$x'$} node{$\bullet$};
    \draw (3.3,0.7) node[below right=-0.05cm] {$x$} node{$\bullet$};
\end{tikzpicture}
}
\hfill
\subfigure[$(x,x')\in I_{\ell-1,k}^2\cup J_{\ell-1}$\label{fig: Cas 5}]{%
\begin{tikzpicture}[scale=0.7]
    \draw[->, very thick] (0,0)--(0,5);
    \draw[->, very thick] (0,0)--(5,0);
    \draw[black,fill=gray!20] (0,0) rectangle (3,3);
    \draw (0,0) rectangle (4.5,4.5);
    \draw[dashed] (3,0)--(3,4.5);
    \draw[dashed] (0,3)--(4.5,3);
    \draw (3,0) node[below] {$t_{\ell-1}^k$};
    \draw (4.5,0) node[below] {$t_{\ell}^k$};
    \draw (0,3) node[left] {$t_{\ell-1}^k$};
    \draw (0,4.5) node[left] {$t_{\ell}^k$};
    \draw (4.2,3.7) node[below left=-0.05cm] {$x'$} node{$\bullet$};
    \draw (3.3,0.7) node[below right=-0.05cm] {$x$} node{$\bullet$};
\end{tikzpicture}     
}
\caption{Notations and the structure of the proof}\label{fig: coupling}
\end{figure}
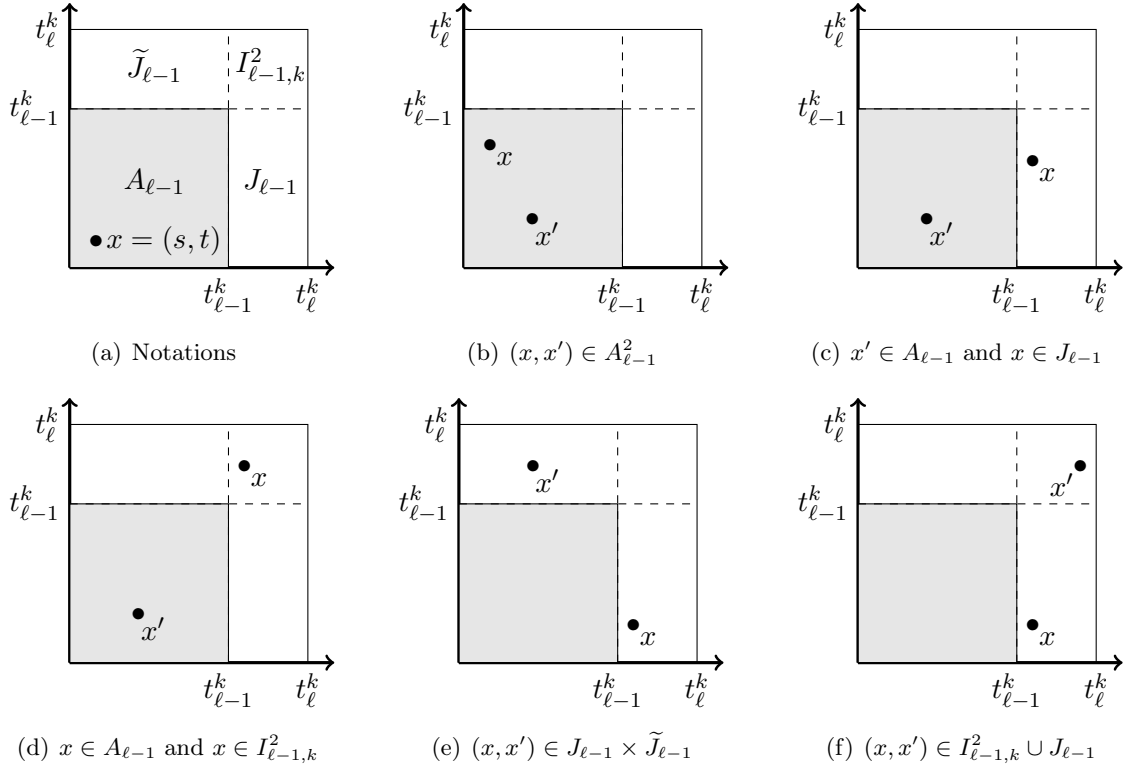

        \begin{itemize}
            \item[$\circ$] If  $(s,t), (s',t') \in A_{\ell-1}$ (this corresponds to Figure \ref{fig: Cas 1}) then $\E{W^B(s',t')W^B(s,t)} = (s\wedge s')(t\wedge t')$.
            \item[$\circ$] Suppose that $(s',t')\in A_{\ell-1}$ and $(s,t)\in A_{\ell}\setminus A_{\ell-1}$ (note that this is illustrate by Figures \ref{fig: Cas 2} and \ref{fig: Cas 3}). We denote by $(i,j)$ the indexes such that $t^k_{i-1} < s \leq t^k_{i}$, $t^k_{j-1} < t \leq t^k_{j-1}$. There are different cases:
            \begin{itemize}
                \item The first case that we handle is the one presented in Figure \ref{fig: Cas 2}. It corresponds to the case where $i=\ell$ and $j < \ell$. Then, we have
                $$W^B\left(s, t^k_{j-1} \right)-W^B\left( t^k_{\ell-1}, t^k_{j-1} \right) = \sum_{p=1}^{j-1} B_{\ell, p} \left(s-t^k_{\ell-1},\frac{1}{k}\right).$$
                And so,
                \begin{align*}
            \E{W^B(s',t')W^B(s,t)} =& \E{ W^B(s',t') B_{\ell,j}\left(s-t^k_{i-1},t-t^k_{j-1} \right)}\\
            &+ \E{W^B(s',t') \left(W^B\left(t^k_{i-1},t \right)+\sum_{p=1}^{j-1} B_{\ell, p} \left(s-t^k_{\ell-1},\frac{1}{k} \right)\right)}.
                \end{align*}
            Since the $B_{\ell,p}$ for $p=1,\dots, j$ are independent of $W^B(s',t')$, we have:
                \begin{align*}
            \E{W^B(s',t')W^B(s,t)} =& \E{W^B(s',t') W^B\left(t^k_{\ell-1},t \right)} = s'(t\wedge t') = (s\wedge s')(t\wedge t').
                \end{align*}         
                \item The case $j=\ell$ and $i<\ell$ can be handled in the same way as the previous case. The computations will lead to
                $$\E{W^B(s',t')W^B(s,t)} = \E{W^B(s',t') W^B\left(s,t^k_{\ell-1} \right)} = (s\wedge s')t' = (s\wedge s')(t\wedge t').$$
                \item We now focus on the case $i=j=\ell$ (see Figure \ref{fig: Cas 3}). The proof of this case relies on the previous case. Indeed, we have
                \begin{align*}
                    \E{W^B(s',t')W^B(s,t)}= & \E{W^B(s',t')W^B\left(s,t^k_{\ell-1} \right)}+\E{W^B(s',t')W^B\left(t^k_{\ell-1},t \right)}\\
                    &- \E{W^B(s',t')W^B\left(t^k_{\ell-1},t^k_{\ell-1} \right)}\\
                    &+ \E{W^B\left(s', t' \right)B_{\ell,\ell}\left(s-t^k_{\ell-1},t-t^k_{\ell-1} \right)}.
                 \end{align*}
                 Hence, the previous cases yields
                \begin{align*}
                    \E{W^B(s',t')W^B(s,t)} &= s't'+s't'-s't'+0 = s't' = (s\wedge s')(t\wedge t').
                 \end{align*}
            \end{itemize}
            \item[$\circ$] Suppose that $(s,t),(s',t') \in A_\ell \setminus A_{\ell-1}$ (see Figures \ref{fig: Cas 4} and \ref{fig: Cas 5}).
            \begin{itemize}
                \item Let us begin with the case illustrated by Figure $\ref{fig: Cas 4}$, i.e. $s'\leq t^k_{\ell-1}< s$. Then there exists $j$ such that $t^k_{j-1}<t\leq t^k_{j}$ and we have
                $$W^B(s,t)-W^B\left(t^k_{\ell-1},t \right) = B_{\ell,j}\left(s-t^k_{\ell-1},t-t^k_{j-1} \right) + \sum_{q=1}^{j-1} B_{\ell,j-q}\left(s-t^k_{\ell-1},\frac{1}{k} \right).$$
                Hence, by independence,
                $$\E{W^B(s',t')\left( W^B(s,t)-W^B\left(t^k_{\ell-1},t \right)\right)} = 0.$$
                and so, using the previous parts,
                $$\E{W^B(s',t')W^B(s,t)} = \E{W^B(s',t')W^B\left(t^k_{\ell-1},t \right)} = s'(t\wedge t')=(s'\wedge s)(t\wedge t').$$
                \item Suppose that we are in the case of Figure \ref{fig: Cas 5}, i.e. $t^k_{\ell-1}<s$ and $t^k_{\ell-1}<s'$. We denote by $j$ and $j'$ the integers such that $t^k_{j-1}< t \leq t^k_{j}$ and $t^k_{j'-1}< t \leq t^k_{j'}$ with $j,j'\leq \ell$. Without loss of generality, we can suppose that $t'\leq t$ (i.e. $j'\leq j$). Then, we get:
                \begin{align*}
                    W^B(s,t)- W^B\left(t^k_{\ell-1},t \right) &= B_{\ell,j} \left(s-t^k_{\ell-1}, t-t^k_{j-1} \right) + \sum_{q=1}^{j-1}B_{\ell,j-q} \left(s-t^k_{\ell-1}, \frac{1}{k} \right),\\
                    W^B(s',t')- W^B\left(t^k_{\ell-1},t' \right) &= B_{\ell,j'} \left(s'-t^k_{\ell-1}, t'-t^k_{j'-1} \right) + \sum_{q=1}^{j'-1}B_{\ell,j'-q} \left(s'-t^k_{\ell-1}, \frac{1}{k} \right).
                \end{align*}
                Hence, by properties of the Brownian sheet, we have
                \begin{align*}
                   & \E{\left( W^B(s,t)- W^B\left(t^k_{\ell-1},t \right)\right)\left(W^B(s',t')- W^B\left(t^k_{\ell-1},t' \right) \right)} \\
                   =& \left\{\begin{array}{cc}
                       \left[s\wedge s'-t^k_{\ell-1}\right]\left[t\wedge t'-t^k_{j-1}\right] + \left[s\wedge s' - t^k_{\ell-1} \right]t^k_{j-1}& \text{if } j=j'  \\
                       \left[ s\wedge s' - t^k_{\ell-1}\right]\left(t'-t^k_{j'-1} \right) + \left[s\wedge s'-t^k_{\ell-1} \right]t^k_{j'-1} &  \text{if } j'<j
                   \end{array} \right.\\
                   =& \left[s\wedge s'-t^k_{\ell-1}\right](t\wedge t').
                \end{align*}

                Moreover, using the previous cases, we have
                $$\hspace*{-1cm}\E{W^B(s,t)W^B\left(t^k_{\ell-1}, t'\right)} = t^k_{\ell-1}(t\wedge t'), \qquad \E{W^B(s',t')W^B\left(t^k_{\ell-1}, t\right)} = t^k_{\ell-1}(t\wedge t')$$
                $$\text{and }\E{W^B\left(t^k_{\ell-1}, t\right)W^B\left(t^k_{\ell-1}, t'\right)} = t^k_{\ell-1}(t\wedge t').$$
                Therefore,
                \begin{align*}
                    \E{W^B(s,t)W^B(s',t')} &= (s'\wedge s)(t\wedge t'). 
                \end{align*}
                \item The case $t^k_{\ell-1}<t$ and $t^k_{\ell-1}<t'$ can be handled using the same type of computations as for the previous case.
            \end{itemize}
        \end{itemize}
    
\end{proof}
\begin{remark}
    Another convenient writing of $W$ on $I_{i,k}\times I_{j,k}$ is given by a linear combination of the $B_{i,j}$. Indeed, for any $(s,t)\in I_{i,k}\times I_{j,k}$, we have
    \begin{align*}
    W(s,t)&= \sum_{1\leq l\leq i-1,1\leq l'\leq j-1}B_{l,l'}(\frac{1}{k},\frac{1}{k}) +B_{i,j}(s-t_{i-1}^k,t-t_{j-1}^k)\\
   &+ \sum_{1\leq l\leq i-1}B_{l,j}(\frac{1}{k},t-t_{j-1}^k)+ \sum_{1 \leq l' \leq j-1}B_{i,l'}(s-t_{i-1}^k, \frac{1}{k})
\end{align*}
\end{remark}

In order to make a coupling between $\Ntil^T$ and $W$, we need to introduce the rectangular increments of $\Ntil^T$ on the discretized $\RR_+^2$:
\begin{equation*}
  \Delta_{(i,j)}^k \Ntil^T := \Ntil^T\left(t_i^k, t_j^k \right)-\Ntil^T\left(t_{i-1}^k, t_j^k \right)-\Ntil^T \left(t_i^k, t_{j-1}^k \right)+\Ntil^T\left(t_{i-1}^k, t_{j-1}^k \right).
\end{equation*}
Note that $\Delta_{(i,j)}^k \Ntil^T$ is a centered Poisson law of parameter $\frac{1}{k^2}$. Moreover, Let $\left(U_{ij}\right)_{(i,j)\in \NN^\ast\times \NN^\ast}$ be standard independent uniforms and take $h_k$ defined by
\begin{equation}
    \label{def: methode fonction inverse}
    h_k(x,u) := F_2^{-1} \left( \mathbb{P}(\zeta_1 < x ) + u \PP(\zeta_1 = x) \right),
\end{equation}
with $\zeta_1 = \Ntil^T\left(1/k, 1/k \right)$ and where $F_2$ is the cumulative distribution function of the random variable $\zeta_2\sim \Ncal \left(0,1/k^2 \right)$. For any $(i,j)\in \NN^\ast\times \NN^\ast$, we set $\xi_{i,j} = h_k \left( \Delta^k_{(i,j)}\Ntil^T, U_{ij} \right)$ and $\beta_{i,j}$ a pinned Brownian sheet on $\left[ 0, \frac{1}{k}\right]$ i.e. $\beta_{i,j}$ is a Gaussian process with covariance function
$$\E{\beta_{i,j}(x,y)\beta_{i,j}(x',y')} = (x\wedge x')(y\wedge y') - k^2 xx'yy', \quad \forall(x,x',y,y')\in \left[ 0, \frac{1}{k}\right]^4.$$
Moreover, the family $\left( \beta_{i,j}\right)_{(i,j)\in \NN^\ast \times \NN^\ast}$ is supposed to be mutually independent and it is also independent of the uniforms $\left(U_{ij}\right)_{(i,j)\in \NN^\ast\times \NN^\ast}$ and of the rectangular increments $\left(\Delta_{(i,j)}^k \Ntil^T\right)_{(i,j)\in \NN^\ast \times \NN^\ast}$.
\begin{remark}
    \label{rem: comonotonicity}
    By taking $\xi_{i,j} = h_k \left( \Delta^k_{(i,j)} \Ntil^T, U_{ij} \right)$ we construct an independent and identically distributed family of Gaussian $\Ncal \left(0, \frac{1}{k^2} \right)$.  Moreover, we also construct a comonotonic coupling between $\xi_{i,j}$ and $\Delta^k_{(i,j)} \Ntil^T$, that is 
\end{remark}

We now define a stochastic process $W \left(= W^{T,k}\right)$ by setting $W\left(0,t \right)=W\left(s,0\right)=0$ for all $(s,t)\in \RR_+^2$ and 
\begin{eqnarray}
   W\left(s,t \right) &=& W\left( s, t_{j-1}^k \right)+W\left( t_{i-1}^k, t \right)-W\left( t_{i-1}^k, t_{j-1}^k \right) \nonumber \\
   &\quad& + \beta_{i,j}\left( s-t_{i-1}^k, t-t_{j-1}^k \right) + k^2 \left( s-t_{i-1}^k\right) \left( t-t_{j-1}^k\right)\xi_{i,j} \label{eq: drap}
\end{eqnarray}
for $t_{i-1}^k < s \leq t_{i}^k$ and $t_{j-1}^k < t \leq t_{j}^k$ with $(i,j) \in \NN^\ast \times \NN^\ast $.

\begin{theorem}
\label{thm: Couplage}
    $W$ as defined in \eqref{eq: drap} is a Brownian sheet on $\RR_+^2$ whose increments over the given grid are coupled with the increments of $\Ntil^T$.
\end{theorem}
    
\begin{proof}
    For any $(i,j)\in \NN^\ast\times \NN^\ast$, we denote by $B_{i,j}$ the following:
    $$B_{i,j}(x,y) = \beta_{i,j}(x,y)+k^2 xy \xi_{i,j}, \quad (x,y)\in \left[ 0, 1/k\right]^2 .$$
    Let us prove that $B_{i,j}$ is a family of independent Brownian sheets.\\
    Take $(x,y)$ and $(x',y')$ both in $\left[ 0, 1/k\right]^2$. Then,
    \begin{align*}
        \E{B_{i,j}(x,y)B_{i,j}(x',y')} &= \E{\beta_{i,j}(x,y)\beta_{i,j}(x',y')} + k^4xyx'y' \E{\xi_{i,j}^2}\\
        &= (x\wedge x')(y\wedge y')-k^2xyx'y' + k^2xyx'y'\\
        &= (x\wedge x')(y\wedge y').
    \end{align*}
    The independence of the family is trivial. 
    Therefore, using Proposition \ref{prop: construction grand drap}, $W$ is a Brownian sheet on $\RR_+^2$. The comonotonicity of the increments is a consequence of Remark \ref{rem: comonotonicity}.
\end{proof}

\begin{notation}
    We will denote by $\FF = \left(\Fcal_t \right)_{t\in [0,1]}$ the following filtration:
    \begin{equation}
    \label{eq: filtration}
       \Fcal_t := \sigma \left( W([0,s]\times [0,a]),\Ntil^T((0,u]\times (0,b])\mid s\leq t, u\leq \frac{\lfloor tk\rfloor}{k}, (a,b)\in \RR_+^2 \right), \quad t\geq 0.
    \end{equation}
\end{notation}

\begin{remark}
\label{rem: filtration}
Ideally, we would have preferred a co-adapted coupling of the Poisson and Brownian martingale measures in the sense of \cite{walsh_introduction_1986}; however, this is unfortunately not the case.  More precisely, if we fix an index $i$ and an index $j$, and consider the process $(W([0,t] \times I_{j,k}))_{t\geq 0} $, we would have liked it to be an $\FF$-Brownian motion, where the filtration is given by  
$$ \Fcal_t = \sigma\left( W([0,u] \times I_{j,k}),~~\Ntil^T([0,s] \times I_{i,k}),~~u\leq t,~~s\leq t\right).$$
Unfortunately, $\left(W([0,t]\times I_{j,k})\right)_{t\geq 0}$ is not a $\FF$-Brownian motion since $W ([s,t]\times I_{j,k})$ is not independent of $\Fcal_s$ for $s\leq t$. By discretizing the coupling, we facilitate our analysis.
\end{remark}

\subsection{Asymptotic quality of the coupling}

In this section, we assess the quality of our coupling by leveraging a result from \cite{rio_upper_2009} also presented in \cite{fournier_simulation_2011}. To establish our bounds, we rely on the previous discretization, which is natural given that we constructed $W$ incrementally.

\begin{lem}[\cite{rio_upper_2009}, Theorem 4.1 or \cite{fournier_simulation_2011}, Theorem 4.1]
\label{lem: maj Rio}
    There is an universal constant $C>0$ such that for any sequence of i.i.d. random variables $(Y_i)_{i\in \NN^*}$ with mean $0$ and variance $\sigma^2$, for any $n \geq 1$,
    $$\Wcal_2^2 \left( \frac{1}{\sqrt{n}} \sum_{i=1}^n Y_i, \Ncal(0,\sigma^2)\right) \leq C\frac{\E{Y_1^4}}{n\sigma^2}.$$
    Here, $\Wcal_2$ stands for the $2$-Wasserstein distance defined by
    $$\Wcal_2^2(\nu_1, \nu_2) = \inf_{\gamma \sim \Gamma(\nu_1, \nu_2)} \iint_{\RR^2} d(x,y)^2 d\gamma(x,y), $$
    where $\Gamma(\nu_1, \nu_2)$ is the set of all couplings of $\nu_1$ and $\nu_2$.
\end{lem}

Since our parameter $T$ is not an integer, we adapt this result in the following lemma in order to obtain a similar upper bound real-valued parameters.
\begin{lem}
There exists a constant $C$ such that for all $a>0$, if $X$ stands for a Poisson random variable with parameter $a$, then 
    \begin{align*}
        \Wcal_2^2 \left(\frac{X-a}{\sqrt{a}}, \Ncal(0,1) \right) \leq \frac{C}{a}.
    \end{align*}
\end{lem}
\begin{proof}
    Let $n\in \NN^\ast$ to be fix latter and $(X_k)_{k=1,\dots,n}$ i.i.d. Poisson random variables with parameter $\frac{a}{n}$. For any $k$, we set $Y_k := \frac{\sqrt{n}}{\sqrt{a}}\left( X_k-\frac{a}{n}\right).$ Then $Y:= \frac{X-a}{\sqrt{a}}$ and $S_n :=\frac{1}{\sqrt{n}}\sum_{k=1}^n Y_k$ have the same distribution. Thus according to Lemma \ref{lem: maj Rio}, there exists a universal constant (independent of $a$ and $n$ such that
     $$\Wcal_2^2 \left( \frac{1}{\sqrt{n}} \sum_{i=1}^n Y_i, \Ncal(0,1)\right) \leq C\frac{\E{Y_1^4}}{n}\leq \frac{C}{a}\left(1+\frac{3a}{n}\right).$$
    Taking $n= \lfloor a \rfloor +1 \geq a$ yields
    $$  \Wcal_2^2 \left( \frac{1}{\sqrt{n}} \sum_{i=1}^n Y_i, \Ncal(0,1)\right) \leq \frac{C}{a}. $$
\end{proof}
\begin{cor}
\label{cor: maj Rio}
    There exists $C>0$ independent of $T$ and $k$ such that for any $(i,j)\in \NN^2$, 
    $$\E{\left| W(I_{i,k}\times I_{j,k} ) -\Ntil^T(I_{i,k}\times I_{j,k}) \right|^2} \leq \frac{C}{T^2} $$
    where $I_{i,k} = \left( t^k_i, t^k_{i+1}\right]$.
\end{cor}
\begin{proof}
    Recall that $\Ntil^T(I_{i,k}\times I_{j,k})= \frac{1}{T} \left( X- \frac{T^2}{k^2}\right) $ 
(in law) where $X$ is a Poisson random variable with parameter $\frac{T^2}{k^2}.$ Thus, by the previous lemma,
$$ \Wcal_2^2 \left( k \Ntil^T(I_{i,k}\times I_{j,k}), \Ncal(0,1)\right)  \leq \frac{Ck^2}{T^2} .$$
Then, since $\frac{1}{k}\Ncal(0,1)$ and $W(I_{i,k}\times I_{j,k} )$ have the same distribution and using the definition of the $\xi_{i,j}$ (which allows us to construct a comonotonic coupling bridging the Wasserstein distance with the $\LL^2$ norm according to Theorem 2.18 in \cite{villani_topics_2003}), we derive 
$$ \E{\left| W(I_{i,k}\times I_{j,k} ) -\Ntil^T(I_{i,k}\times I_{j,k}) \right|^2} = \Wcal_2^2 \left( \Ntil^T(I_{i,k}\times I_{j,k}), W(I_{i,k}\times I_{j,k}) \right)  \leq \frac{C}{T^2} .$$
\end{proof}

According to Remark \ref{rem: filtration}, we introduce the following notation:

\begin{notation}
\label{not: discretized}
    For any process $f^T$, we will denote by $\bar{f}^T$ its càg-làd discretized version along the subdivision $\left( t^k_i\right)_{i=0,\dots, k}=\left( \frac{i}{k}\right)_{i=0,\dots, k}$, that is:
    \begin{equation}
    \label{eq: def lambda discret}
    \bar{f}^T_t := \sum_{i=0}^{k-1} \1_{I_{i,k}}(t)f^T_{t_i^k} .
    \end{equation}
\end{notation}
Note that the discretized version of any $\FF^T$-predictable process is a $\FF$-predictable process with $\FF^T= (\Fcal^N_{tT})_{t\geq 0}$ the natural history of $N^T$ and $\FF$ as defined in \eqref{eq: filtration}.
\begin{prop}\label{prop:maj-intw-nT}
    Let $f\in \LL^{\infty}(0,1)$ and $(u^T_s)_{s\in [0,1]}$ be a positif $\FF^T$-predictable process, such that:
    $$\sup_{s\in [0,1]}\E{u^T_s}<+\infty . $$
    Then, there exists $C>0$ such that for any $T>1$,
    \begin{align*}
    \hspace*{-1cm}&\E{\left|\iint_{(0,1]\times\RR_+} \bar{f}(s) \1_{\theta \leq \bar{u}^T_s}\left[W(ds,d\theta ) -\Ntil^T(ds,d\theta)\right] \right|^2} \leq \left\| f\right\|_{\LL^{\infty}(0,1)}^2 \left(\frac{1}{k} + \frac{1}{T^2} \left(k^2+3 \right)  \sup_{s \in [0,1]}\E{ u^T_s} \right).
    \end{align*}
    where $\bar{f}$ and $\bar{u}^T$ stand for the $\FF$-predictable (with $\FF$ defined in \eqref{eq: filtration}) discretized version of $f$ and $u^T$.
\end{prop}

\begin{proof}
    Let $U^T$ and $U^T_k$ the predictable processes defined by
    \begin{equation}
    \label{def: U^T et U^T_k}
        U^T(s,\theta):=\bar{f}(s)\1_{\theta \leq \bar{u}^T_s}\quad \text{and} \quad U^T_k(s,\theta):=\sum_{i=0}^k \sum_{j=0}^{\infty} f(t^k_i) \1_{I_{i,k}\times I_{j,k}}(s,\theta) \1_{t_{j}^k \leq u_{t_i^k}^T}.
    \end{equation}
    For $i\in \{ 0,\dots,k-1 \}$, we define by $j_i^\ast$ the index of the last time $t_j^k$ such that $t_j^k \leq u^T_{t_i^k}$, i.e.
    \begin{align*}
    j_i^\ast=\max\{j\in \{ 0,\dots,k-1 \} \mid t_j^k \leq u^T_{t_i^k}\}.
    \end{align*}
    Since $u^T$ is a predictable process,  $j_i^\ast$ is a $\Fcal_{t_i^k}$ (see \eqref{eq: filtration}) random variable with integer values. Moreover, by definition, we have the following inequality:
        $$t_{j^*_i}^k \leq u^T_{t_i^k} < t_{j^*_i}^k + k^{-1}.$$
    Hence, we can rewrite $U^T_k$ using $j_i^\ast$ as below
    \begin{equation}
    \label{def: U^T_k}
        U^T_k(s,\theta)=\sum_{i=0}^k f(t^k_i)\1_{I_{i,k}\times (0,t_{j^*_i +1}^k]}(s,\theta).
    \end{equation}
    Using the definition of $U^T_k$ and $U^T$, we have
    \begin{align*}
        &\E{\left|\iint_{(0,1]\times\RR_+}\bar{f}^T(s) \1_{\theta \leq \bar{u}^T_s}\left[W(ds,d\theta ) -\Ntil^T(ds,d\theta)\right] \right|^2}\\
        &\leq C  \E{\left|\iint_{(0,1]\times\RR_+}\left[U^T(s) -U^T_k(s,\theta)\right] W(ds,d\theta)\right|^2}\\
        &+C \E{\left|\iint_{(0,1]\times\RR_+}\left[U^T(s) -U^T_k(s,\theta)\right] \Ntil^T(ds,d\theta)\right|^2}\\
        &+ C \E{\left|\iint_{(0,1]\times\RR_+}U^T_k(s,\theta) \left[W(ds,d\theta ) -\Ntil^T(ds,d\theta)\right]\right|^2}.
    \end{align*}
    And so,
    \begin{align}
        &\E{\left|\iint_{(0,1]\times\RR_+} \1_{\theta \leq \bar{u}^T_s}\left[W(ds,d\theta ) -\Ntil^T(ds,d\theta)\right] \right|^2}\nonumber \\
        &\leq C  \E{\int_0^1\int_{\RR_+}\left|U^T(s) -U^T_k(s,\theta)\right|^2 d\theta ds}\label{eq: maj_couplage} \\
        &+ C \E{\left|\iint_{(0,1]\times\RR_+} U^T_k(s,\theta) \left[W(ds,d\theta ) -\Ntil^T(ds,d\theta)\right]\right|^2}. \nonumber
    \end{align}
    On the one hand, we have:
    \begin{align*}
        &\E{ \int_0^1 \int_{\RR_+}\left|U^T(s,\theta)-U^T_k(s,\theta)\right| ^2 d\theta ds}\\
        =&\sum_{i=0}^{k-1} \left|f(t^k_i) \right|^2 \int_{t_i^k}^{t_{i+1}^k}\int_{\RR_+}\E{\left|\1_{\theta \leq u^T_{t^k_i}}- \1_{I_i^k\times [0,t^k_{j^*_i +1}]}(s,\theta)\right|^2 } d\theta ds .
    \end{align*}
    Thus, 
    \begin{align*}
     \E{ \int_0^1 \int_{\RR_+}\left|U^T(s,\theta)-U^T_k(s,\theta)\right| ^2 d\theta ds} &\leq \left\|f \right\|_{\LL^{\infty}(0,1)}^2\sum_{i=0}^{k-1} \int_{t_i^k}^{t_{i+1}^k} \E{\left|u^T_{t_i^k}-t^k_{j^*_i +1}\right|} ds\\
     &\leq \left\|f \right\|_{\LL^{\infty}(0,1)}^2 \sum_{i=0}^{k-1} \int_{t_i^k}^{t_{i+1}^k} \left( \E{\left|u^T_{t_i^k}-t^k_{j^*_i}\right|} +k^{-1}\right) ds.
    \end{align*}
    Then, by definition of $t_{j^*_i}^k$, we have:
    \begin{equation}
    \label{eq: couplage_maj_1}
     \E{ \int_0^1 \int_{\RR_+}\left|U^T(s,\theta)-U^T_k(s,\theta)\right| ^2 d\theta ds} \leq  C\frac{\left\|f \right\|_{\LL^{\infty}(0,1)}^2}{k}.
    \end{equation}
    One the other hand, using \eqref{def: U^T_k}, we observe that
    \begin{align*}
        &\hspace*{-1.5cm}\E{\left|\iint_{I_{i,k}\times \RR_+}  U^T_k(s,\theta)\left(W(ds,d\theta ) -\Ntil^T(ds,d\theta)\right)\right|^2} =\E{\left|\iint_{I_{i,k}\times (0, t_{j_i^\ast}^k]}  f(t^k_i)\left[W(ds,d\theta ) -\Ntil^T(ds,d\theta)\right]\right|^2}.
    \end{align*}
    Hence,
    \begin{align*}
    \hspace*{-1.3cm}\E{\left|\iint_{I_{i,k}\times \RR_+}  U^T_k(s,\theta)\left(W(ds,d\theta ) -\Ntil^T(ds,d\theta)\right)\right|^2} &= \E{ \left|f(t^k_i) \sum_{\ell=0}^{j_i^\ast} W\left(I_{i,k}\times I_{\ell,k}\right) - \Ntil^T\left(I_{i,k}\times I_{\ell,k}\right) \right|^2 }.
    \end{align*}
    Using the fact that the rectangular increments $\left(W\left(I_{i,k}\times I_{\ell,k}\right) - \Ntil^T\left(I_{i,k}\times I_{\ell,k}\right) \right)_{(i,\ell)\in \NN^2}$ are independent, we have
    \begin{align*}
        \hspace*{-1.5cm}\E{\left|\iint_{I_{i,k}\times \RR_+} U^T_k(s,\theta) \left(W(ds,d\theta ) -\Ntil^T(ds,d\theta)\right)\right|^2} &\leq \left\| f\right\|_{\LL^{\infty}(0,1)}^2\E{\sum_{\ell=0}^{j_i^*}\left|W\left(I_{i,k}\times I_{\ell,k}\right) - \Ntil^T\left(I_{i,k}\times I_{\ell,k}\right) \right|^2}.
    \end{align*}
    Moreover, since $j_i^*$ is $\Fcal_{t_i^k}$ measurable and since $\left( W\left(I_{i,k}\times I_{\ell,k}\right) - \Ntil^T\left(I_{i,k}\times I_{\ell,k}\right)\right)_{\ell\in \NN}$ is independent of $\Fcal_{t_i^k}$, we have
    \begin{align*}
        \E{\sum_{\ell=0}^{j_i^*}\left|W\left(I_{i,k}\times I_{\ell,k}\right) - \Ntil^T\left(I_{i,k}\times I_{\ell,k}\right) \right|^2} &= \E{\sum_{\ell=0}^{j_i^*}\EE_{t_i^k}\left[\left|W\left(I_{i,k}\times I_{\ell,k}\right) - \Ntil^T\left(I_{i,k}\times I_{\ell,k}\right) \right|^2\right]}\\
        &=\E{\sum_{\ell=0}^{j_i^*}\E{\left|W\left(I_{i,k}\times I_{\ell,k}\right) - \Ntil^T\left(I_{i,k}\times I_{\ell,k}\right) \right|^2}}.
    \end{align*}
    Hence, by Corollary \ref{cor: maj Rio}, we compute
        \begin{equation}
        \E{\sum_{\ell=0}^{j_i^*}\left|W\left(I_{i,k}\times I_{\ell,k}\right) - \Ntil^T\left(I_{i,k}\times I_{\ell,k}\right) \right|^2} \leq \frac{C}{T^2} \E{j_i^*} .
    \end{equation}
    From the very definition of $j_i^*$ one have 
        $$ t_{j^*_i}^k \leq  u^T_{t_i^k} \leq t_{j^*_i}^k + k^{-1} \quad \text{and} \quad j^*_i \leq k u^T_{t_i^k}.$$
    Thus, 
  \begin{align*}
        \E{\left|\iint_{I_{i,k}\times \RR_+} U^T_k(s,\theta)  \left(W(ds,d\theta ) -\Ntil^T(ds,d\theta)\right)\right|^2} &\leq C\frac{\left\| f\right\|_{\LL^{\infty}(0,1)}^2}{T^2} k \E{u^T_{t_i^k}} \\
        &\leq C\frac{\left\| f\right\|_{\LL^{\infty}(0,1)}^2}{T^2} k  \sup_{s \in [0,1]}\E{u^T_s} .
    \end{align*}
    Adding these inequalities for $i=0,\dots, k-1$ and using the fact that the random variables\\ $ \left(\iint_{I_{i,k}\times \RR_+} U^T_k(s,\theta) \left[W(ds,d\theta ) -\Ntil^T(ds,d\theta)\right]\right)_{i=0,\dots, k-1}$ are independent, we proved that
    \begin{equation}
    \label{eq: couplage_maj_2}
          \E{\left|\iint_{(0,1]\times \RR_+}U^T_k(s,\theta) \left(W(ds,d\theta ) -\Ntil^T(ds,d\theta)\right)\right|^2} 
          \leq C\frac{\left\| f\right\|_{\LL^{\infty}(0,1)}^2}{T^2} k^2  \sup_{s \in [0,1]}\E{ u^T_s}.
    \end{equation}
    Gathering \eqref{eq: couplage_maj_1} and \eqref{eq: couplage_maj_2} in \eqref{eq: maj_couplage}, we finally have:
        \begin{align*}
           \E{\left|\iint_{(0,1]\times \RR_+} f^T(s)\1_{\theta \leq \bar{u}^T_s}\left[W(ds,d\theta ) -\Ntil^T(ds,d\theta)\right] \right|^2} &\leq C \left\| f\right\|_{\LL^{\infty}(0,1)}^2 \left(\frac{1}{k} + \frac{k^2}{T^2}   \sup_{s \in [0,1]}\E{ u^T_s} \right).
        \end{align*}
\end{proof}

We end this section with Theorem \ref{thm: maj_sup_int_NTW}. In the latter, we control the supremum norm of stochastic integrals which appear in Proposition \ref{prop:maj-intw-nT}.
\begin{theorem}
\label{thm: maj_sup_int_NTW}
Let $f$ be a differentiable function with continuous derivative on $[0,1]$ and $u^T$ be a positive predictable process for the filtration $\FF^{T}$ such that 
$$\sup_{s\in [0,1]}\E{\left|u^T_s \right|^2}<C.$$
Then, there exists $C>0$ such that for any $T>1$,
    \begin{equation}
        \E{\sup_{t\in [0,1]} \left|\iint_{(0,t]\times \RR_+} f(s) \1_{\theta \leq \bar{u}^T_s} \left(\Ntil^T-W \right)(ds,d\theta) \right|^2 } \leq C \left(\frac{1}{k}+ \frac{k^2}{T^2} \right).
    \end{equation}
\end{theorem}
\begin{proof}
    First, let us notice that 
    \begin{align*}
        &\sup_{t\in [0,1]} \left|\iint_{(0,t]\times \RR_+} f(s) \1_{\theta \leq \bar{u}^T_s} \left(\Ntil^T-W \right)(ds,d\theta) \right|^2 \\
        \leq& \left(\sum_{i=0}^{k-1} \sup_{t\in (t_i^k,t^k_{i+1}]} \left|\iint_{(t_i^k, t] \times \RR_+} f(s) \1_{\theta \leq \bar{u}^T_s} \left(\Ntil^T-W \right)(ds,d\theta) \right|^4\right)^{1/2} \\
        &+ \max_{i=1, \dots,k} \left|\iint_{(0, t^k_i] \times \RR_+} f(s) \1_{\theta \leq \bar{u}^T_s} \left(\Ntil^T-W \right)(ds,d\theta) \right|^2.
    \end{align*}
    Let us deal the terms separately. To do so, we denote for each $i$ by $A^{(i)}_1$ and $A^{(i)}_2$ the followings:
    \begin{align*}
        A^{(i)}_1 &:= \sup_{t\in (t_i^k,t^k_{i+1}]} \left|\iint_{(t_i^k, t] \times \RR_+} f(s) \1_{\theta \leq u^T_{t^k_i}} \left(\Ntil^T-W \right)(ds,d\theta) \right|^4, \\
        A^{(i)}_2 &:= \left|\iint_{(0, t^k_i] \times \RR_+} f(s) \1_{\theta \leq \bar{u}^T_s} \left(\Ntil^T-W \right)(ds,d\theta) \right|^2.
    \end{align*}
    \paragraph{Control of \texorpdfstring{$\E{A^{(i)}_1}$}{}}~\\
    Bu using the fact that $(a+b)^4 \leq C(a^4+b^4)$, we have:
    \begin{align*}
       \E{A^{(i)}_1} \leq & C \E{\sup_{t\in (t_i^k, t_{i+1}^k]} \left| \iint_{(t_i^k, t] \times \RR_+} f(s) \1_{\theta \leq u^T_{t^k_i}} \Ntil^T(ds,d\theta)\right|^4} \\
       &+ C\E{\sup_{t\in (t_i^k, t_{i+1}^k]} \left| \iint_{(t_i^k, t] \times \RR_+} f(s) \1_{\theta \leq u^T_{t^k_i}} W(ds,d\theta)\right|^4}.
    \end{align*}
    Hence, Doob's inequality yields
        \begin{align*}
       \hspace*{-1cm}\E{A^{(i)}_1} &\leq C \E{ \left| \iint_{(t_i^k, t_{i+1}^k] \times \RR_+} f(s) \1_{\theta \leq u^T_{t^k_i}} \Ntil^T(ds,d\theta)\right|^4} + C\E{\left| \iint_{(t_i^k, t_{i+1}^k] \times \RR_+} f(s) \1_{\theta \leq u^T_{t^k_i}} W(ds,d\theta)\right|^4}.
    \end{align*}
    Thus, by Lemma \ref{lem-outil-tilde-N-T}, we obtain
    \begin{align*}
       \E{A^{(i)}_1} &\leq C \E{ \left| \int_{t^k_i}^{t^k_{i+1}} f(s)^2 u^T_{t^k_i}ds\right|^2}+\frac{C}{T^2} \int_{t^k_i}^{t^k_{i+1}} f(s)^4 \E{u^T_{t^k_i}}ds + C\E{\left| \int_{t^k_i}^{t^k_{i+1}} f(s)^2 u^T_{t^k_i} ds \right|^2}.
    \end{align*}
    So, by using Jensen's inequality combined with $\sup_{t\in[0,1]} \E{\left|u^T_t \right|^2}\leq C$ and $\left\| f\right\|_{\LL^{\infty}(0,1)}<+\infty$, we get:
    \begin{align*}
       \E{A^{(i)}_1} &\leq \frac{C}{k} \int_{t^k_i}^{t^k_{i+1}} f(s)^4 \E{\left|u^T_{t^k_i}\right|^2} ds+\frac{C}{T^2} \int_{t^k_i}^{t^k_{i+1}} f(s)^4 \E{u^T_{t^k_i}}ds + \frac{C}{k}\int_{t^k_i}^{t^k_{i+1}} f(s)^4 \E{\left|u^T_{t^k_i}\right|^2} ds \\
       &\leq C\frac{\sup_{t\in [0,1]} \E{\left|u^T_t \right|^2}}{k} \int_{t^k_i}^{t^k_{i+1}} f(s)^4 ds+C\frac{\sup_{t\in [0,1]} \E{\left|u^T_t \right|} }{T^2} \int_{t^k_i}^{t^k_{i+1}} f(s)^4 ds.\\
       &\leq C\left\| f\right\|_{\LL^\infty (0,1)}^4 \left( \frac{\sup_{t\in [0,1]} \E{\left|u^T_t \right|^2}}{k^2} + \frac{\sup_{t\in [0,1]} \E{\left|u^T_t \right|}}{T^2k}\right).
    \end{align*}  
    Thanks to this last inequality, we compute
    \begin{align}
        &\left(\sum_{i=0}^{k-1} \sup_{t\in (t_i^k,t^k_{i+1}]} \left|\iint_{(t_i^k, t] \times \RR_+} f(s) \1_{\theta \leq u^T_{t^k_i}} \left(\Ntil^T-W \right)(ds,d\theta) \right|^4\right)^{1/2} \nonumber\\
        &\leq C\left\| f\right\|_{\LL^\infty (0,1)}^2 \left( \frac{\sup_{t\in [0,1]} \E{\left|u^T_t \right|^2}}{k} + \frac{\sup_{t\in [0,1]} \E{\left|u^T_t \right|}}{T^2}\right)^{1/2}.
    \end{align}
    \paragraph{Control of \texorpdfstring{$\E{\max_{i=1,\dots k}A^{(i)}_2}$}{}} ~\\
    To deal with this term, we need to split it in two as follows:
    \begin{align*}
        A^{(i)}_2 \leq& C\left|\iint_{(0, t^k_i] \times \RR_+} \left(f(s) - \sum_{j=0}^{k-1} f(t^k_j) \1_{s\in I_{j,k}} \right) \1_{\theta \leq \bar{u}^T_s} \left(\Ntil^T - W \right)(ds,d\theta)\right|^2 \\
        &+ C \left|\iint_{(0, t^k_i] \times \RR_+} \sum_{j=0}^{k-1} f(t^k_j) \1_{s\in I_{j,k}} \1_{\theta \leq \bar{u}^T_s} \left(\Ntil^T - W \right)(ds,d\theta) \right|^2.
    \end{align*}
    Hence, 
        \begin{align*}
        A^{(i)}_2 \leq& C\left|\iint_{(0, t^k_i] \times \RR_+} \left(f(s) - \sum_{j=0}^{k-1} f(t^k_j) \1_{s\in I_{j,k}} \right) \1_{\theta \leq \bar{u}^T_s} \left(\Ntil^T - W \right)(ds,d\theta)\right|^2 \\
        &+ C\left| \sum_{j=0}^{i-1}  f(t^k_j) \iint_{(t^k_j, t^k_{j+1}] \times \RR_+}   \1_{\theta \leq u^T_{t^k_j}} \left(\Ntil^T - W \right)(ds,d\theta) \right|^2.
    \end{align*}
    We now make use of $(a-b)^2 \leq C(a^2+b^2)$ for the first term to obtain:
    \begin{align*}
        A^{(i)}_2 \leq& C\left|\iint_{(0, t^k_i] \times \RR_+}\left(f(s) - \sum_{j=0}^{k-1} f(t^k_j) \1_{s\in I_{j,k}} \right) \1_{\theta \leq \bar{u}^T_s} \Ntil^T(ds,d\theta)\right|^2\\
        &+C\left|\iint_{(0, t^k_i] \times \RR_+} \left(f(s) - \sum_{j=0}^{k-1} f(t^k_i) \1_{s\in I_{j,k}} \right) \1_{\theta \leq \bar{u}^T_s} W(ds,d\theta)\right|^2 \\
        &+ C 
        \left|\sum_{j=0}^{i-1}   f(t^k_j)\iint_{(t^k_j, t^k_{j+1}] \times \RR_+}    \1_{\theta \leq u^T_{t^k_j}} \left(\Ntil^T - W \right)(ds,d\theta) \right|^2.
    \end{align*}
    We denote by $A^{(i)}_{2,1}$, $A^{(i)}_{2,2}$ and $A^{(i)}_{2,3}$ the three right-hand terms which we handle separately. Let us begin with $A^{(i)}_{2,1}$. Since $\left( \int_{0}^{t^k_i} \int_{\RR_+} \left(f(s) - \sum_{j=0}^{k-1} f(t^k_j) \1_{s\in I_{j,k}} \right) \1_{\theta \leq \bar{u}^T_s} \Ntil^T(ds,d\theta)\right)_{i=0,\dots, k-1}$ is a discrete martingale, we can make use of Doob's inequality and we get
    \begin{align*}
        \E{\max_{i=1, \dots, k} A^{(i)}_{2,1}} &\leq C\E{\left|\iint_{(0,1]\times \RR_+} \left(f(s) - \sum_{j=0}^{k-1} f(t^k_j) \1_{s\in I_{j,k}} \right) \1_{\theta \leq \bar{u}^T_s} \Ntil^T(ds,d\theta)\right|^2}\\
        &\leq C \int_{0}^{1} \left|f(s) - \sum_{j=0}^{k-1} f(t^k_j) \1_{s\in I_{j,k}} \right|^2 \E{\bar{u}^T_s} ds
    \end{align*}
    Hence, by the Mean Value Inequality, 
    \begin{align}
    \label{eq: A21}
        \E{\max_{i=1, \dots, k} A^{(i)}_{2,1}} &\leq C \sum_{j=0}^{k-1}\int_{t_j^k}^{t_{j+1}^k} \left|f(s) -  f(t^k_j)  \right|^2 \E{u^T_{t^k_j}} ds \nonumber \\
        &\leq C \left\| f' \right\|_{\LL^{\infty}(0,1)}^2 \sup_{t\in [0,1]} \E{u^T_t} \frac{1}{k^2}.
    \end{align}
    Here we have used $\sup_{t\in[0,1]} \E{u^T_t} \leq C$ and $\left\|f' \right\|_{\LL^{\infty}(0,1)}<+\infty$.
    Similarly, we can obtain
    \begin{align}
    \label{eq: A22}
        \E{\max_{i=1, \dots, k} A^{(i)}_{2,2}} &\leq C \left\| f' \right\|_{\LL^{\infty}(0,1)}^2 \sup_{t\in [0,1]} \E{u^T_t} \frac{1}{k^2}.
    \end{align}
    On the other hand, since $\left(\sum_{j=0}^{i-1}  f(t^k_j) \int_{t^k_j}^{t^k_{j+1}} \int_{\RR_+}   \1_{\theta \leq u^T_{t^k_j}} \left(\Ntil^T - W \right)(ds,d\theta) \right)_{i=1, \dots, k}$ is a discrete martingale, we have
    \begin{align*}
    \E{\max_{i=1, \dots, k} A^{(i)}_{2,3}} &\leq C \E{\left|\sum_{j=0}^{k-1} f(t^k_j)  \iint_{(t^k_j, t^k_{j+1}] \times \RR_+}    \1_{\theta \leq u^T_{t^k_j}} \left(\Ntil^T - W \right)(ds,d\theta)\right|^2}\\
    &\leq C \E{ \left|\iint_{(0,1]\times \RR_+}  \bar{f}(s) \1_{\theta \leq \bar{u}^T_s} \left(\Ntil^T - W \right)(ds,d\theta)\right|^2}
    \end{align*}
    By Proposition \ref{prop:maj-intw-nT}, we conclude that
    \begin{align}
    \label{eq: A23}
        \E{\max_{i=1, \dots, k} A^{(i)}_{2,3}} &\leq C \left\| f\right\|_{\LL^{\infty}(0,1)}^2 \left(\frac{1}{k} + \frac{k^2}{T^2}  \sup_{s \in [0,1]}\E{ u^T_s} \right).
    \end{align}
    Regrouping \eqref{eq: A21}, \eqref{eq: A22} and \eqref{eq: A23} we deduce that:
    $$\E{\max_{i=1,\dots, k} A^{(i)}_2} \leq C \left( \frac{1}{k} + \frac{k^2}{T^2} \right)$$
\end{proof}

\section{Upper bound of the 2-Wasserstein distance}
\label{sec: main upper bound}

In this section, we present the convergence rate of the different considered Hawkes processes to their limit. We will make use of general notations for the processes by writing:
\begin{align*}
    \Lambda^{T,\natural}_t &= \frac{\mu}{T}+ \mu\int_0^t \Psi^{(T),\natural}(t-s) ds + \iint_{(0,t)\times \RR_+} \Psi^{(T),\natural}(t-s) \1_{\theta \leq \Lambda^{T,\natural}_s} \Ntil^T(ds,d\theta),\\
    X^\natural_t &= \mu \int_0^t \rho^\natural(t-s) ds + \iint_{(0,t)\times \RR_+} \rho^\natural(t-s) \1_{\theta \leq X^\natural_s} W(ds,d\theta),
\end{align*}
where $\natural\in \{ -,0,+\}$ and where $\rho^\natural$ are defined according to Table \ref{tab: notation}.

In addition, since the coupling depends on parameter $k$, we decide to fix the value of $k$ in this section by taking $$k=\lfloor T^{4/5}\rfloor+1.$$
We now state our main theorem:
\begin{theorem}
\label{th: NUHP to X}
    Under Assumption \ref{assump: phi}, there exists $C>0$ such that for any $T> 2$,
    \begin{equation*}
        \E{\sup_{t\in[0,1]} \left|\Lambda^{T,\natural}_t - X^\natural_t\right|^2} \leq \frac{C}{\ln(T)}
    \end{equation*}
    with $\natural\in \{ -,0,+\}$.
\end{theorem}

Before writing the proof, let us write some remarks and some notations employed in the proof.
\begin{remark}
    Note that we do not make use of the assumption $\sup_{T>1} \left\|\Psi^{T,\natural} \right\|_{\infty}<+\infty$. Indeed, in our proof, we only need $\sup_{T>1} \left\|\Psi^{(T),\natural} \right\|_{\LL^\infty(0,1)}<+\infty$ which is a result that we have proved in Lemma \ref{lem: sup_norm_Psi}. Consequently, our result extends the original convergence (non-quantified) result of \cite{jaisson_limit_2015}.
\end{remark}

\begin{remark}
\label{rem: skorokhod}
    Whereas the result is proved using the $2$-Wasserstein distance between $\Ntil^T$ and $W$, the distance chosen in Theorem \ref{th: NUHP to X} is the supremum norm. This allows us to characterize the convergence rate in the Skorokhod space $D$ (see Billinglsey \cite{billingsley_convergence_1968} for more detail on such space).
\end{remark}
\begin{cor}
    There exists $C>0$ such that for any $T>2$,
    \begin{equation*}
        \E{d_D (\Lambda^{T,\natural}, X^\natural )^2} \leq \frac{C}{\ln(T)}.
    \end{equation*}
    where $d_D$ stand for the Skorokhod distance and $\natural\in \{ -,0,+\}$.
\end{cor}

\begin{proof}[Proof of Theorem \ref{th: NUHP to X}]
    For the sake of simplicity, we drop out the index $\natural$ for all the processes.\\
    For $t\geq 0$, we denote by $\Delta^T_t$ (resp. $d^T(t)$) the difference between $\Lambda^T$ and $X$ (resp. $\Psi^{(T)}$ and $\rho$) at time $t$, that is: 
    $$\Delta^T_t := \Lambda^T_t-X_t \quad (\text{resp. } d^T(t) := \Psi^{(T)}(t)-\rho(t)\,)$$
    where $\Psi^{(T)}$ is defined in \eqref{eq:def_psiT}. We also define $\hat{\Delta}^T$ by 
    $$\hat{\Delta}^T := \bar{\Lambda}^T-X$$
    where $\bar{\Lambda}^T$ is the discretized version of $\Lambda^T$ introduced in Notation \ref{not: discretized}.
    We divide the proof in three main steps: in the first part of the proof, we rewrite $\Delta^T$ in a convenient way and we introduce different notations; then, for any $t\in [0,1]$, we control $\sup_{s\in [0,t]} \E{\left|\hat{\Delta}^T_s\right|}$ with a Yamada's function; and finally we provide an upper bound for $\E{\sup_{s\in [0,t]} \left| \Delta^T_s\right|^2}$.

    \paragraph{Rewriting of \texorpdfstring{$\Delta^T$}{}.}~\\
    We fix $t\geq 0$. Then, using \eqref{def: Lambda^T} and $\rho(t)=m\rho(t-\cdot)\rho(\cdot)$, we obtain
    \begin{align}
    \Delta_t^T&=\frac{\mu}{T}+ \mu \int_0^t \left[\Psi^{(T)}(u) -\rho(u)\right]du + \iint_{(0,t)\times\RR_+}\left[\Psi^{(T)}(t-s) -\rho(t-s)\right]\1_{\theta \leq \Lambda_s^T}\Ntil^T(ds,d\theta) \nonumber \\
    &+ \iint_{(0,t)\times\RR_+}\rho(t-s)\left[\1_{\theta \leq \Lambda_s^T} -\1_{\theta \leq \bar{\Lambda}_s^T}\right]\Ntil^T(ds,d\theta) \nonumber \\
    &+ \iint_{(0,t)\times\RR_+}\rho(t-s)\1_{\theta \leq \bar{\Lambda}_s^T}(\Ntil^T-W)(ds,d\theta)\nonumber \\
    &+ \iint_{(0,t)\times\RR_+}\rho(t-s)\left[\1_{\theta \leq \bar{\Lambda}_s^T}-\1_{\theta \leq X_s}\right]W(ds,d\theta) \nonumber\\
    &=: \mu \int_0^t d^T(u) du + Y^T_{t} + m\rho(t) M^{\Lambda^T, \Ntil^T}_{t^-} +m\rho(t)M^{\Ntil^T-W}_{t^-}+ m\rho(t)M^{\hat{\Delta}^T,W}_{t^-} \label{eq: Delta^T avec reste}
    \end{align}
 where $Y^T$, $M^{\Lambda^T, \Ntil^T}$, $M^{\Ntil^T-W}$ and $M^{\hat{\Delta}^T, W}$ are defined by
    \begin{align*}
    Y^T_t&:= \iint_{(0,t)\times\RR_+}\left[\Psi^{(T)}(t-s)) -\rho(t-s)\right]\1_{\theta \leq \Lambda_s^T}\Ntil^T(ds,d\theta), \\
    M^{\Lambda^T, \Ntil^T}_t &:=\iint_{(0,t]\times\RR_+}\rho(-s)\left[\1_{\theta \leq \Lambda_s^T}-\1_{\theta \leq \bar{\Lambda}^T_s}\right]\Ntil^T(ds,d\theta),\\
    M^{\Ntil^T-W}_t &:= \iint_{(0,t]\times\RR_+} \rho(-s) \1_{\theta \leq \bar{\Lambda}^T_s} (\Ntil^T-W)(ds, d\theta),\\
    M^{\hat{\Delta}^T,W}_t &:=\iint_{(0,t]\times\RR_+} \rho(-s)\left[\1_{\theta \leq \bar{\Lambda}_s^T}-\1_{\theta \leq X_s}\right]W(ds,d\theta).
    \end{align*}
This allows us to rewrite the previous equation as
\begin{equation}
\label{eq: Delta_final}
    \Delta^T_t = R^T_t + m \rho(t)M^{\hat{\Delta}^T,W}_{t^{-}}
\end{equation}
where $R^T_t$ is naturally introduced to guarantee the validity of the equation.

\paragraph{Upper bound of \texorpdfstring{$\sup_{s\in [0,t]} \E{\left|\hat{\Delta}^T_s\right|}$}{}}. ~\\
The purpose of this paragraph is to control the expectation of $\hat{\Delta}^T_t = \bar{\Lambda}^T_t-X_t$. To keep the length of this step within limits we postpone some computations to Section \ref{subs: tools for theorem NUHP to X}. We focus here on the central part of the proof that is the use of Yamada-type function to derive the inequality. Firstly, note that
\begin{equation*}
    \E{\left|\hat{\Delta}^T_t \right|} \leq \E{\left|\Lambda^T_t - \bar{\Lambda}^T_t\right|}+\E{\left|R^T_t\right|} + C\E{\left|M^{\hat{\Delta}^T, W}_{t^{-}}\right|},~~t\in [0,1].
\end{equation*}
Let $(\eps,\eta)\in \RR_+^\ast \times \RR_+^\ast$ and $s\in [0,t]$ be fixed. We define by $\Upsilon_{\eps, \eta}$ the Yamada's function which is a $\Ccal^2$ function that satisfies, for any $x\in \RR$, the followings
\begin{equation}
\label{eq: yamada}
    \begin{array}{cc}
        |x|-\eps \leq \Upsilon_{\eps, \eta}(x) \leq |x|, &\quad \left| \Upsilon'_{\eps, \eta}(x)\right| \leq 1,\\
        \left| \Upsilon''_{\eps, \eta}(x)\right| \leq \frac{2m^2}{|x|\eta}, & \quad  \left| \Upsilon''_{\eps, \eta}(x)\right| \leq 2m^2\frac{e^{\eta/m^2}}{\eps \eta}.
    \end{array}
\end{equation}
Such function exists and its construction is originally presented in \cite{yamada_sur_1978}; however the methodology of the proof is inspired by \cite{alfonsi_discretization_2005}.
Then, we have:
\begin{equation}
\label{eq: delta_yamada}
    \E{\left|\hat{\Delta}^T_s \right|} \leq \E{\left|\Lambda^T_s - \bar{\Lambda}^T_s\right|} + \E{\left|R^T_s\right|} + C\left(\eps +\E{\Upsilon_{\eps,\eta}\left(M^{\hat{\Delta}^T, W}_s\right) } \right).
\end{equation}
Let us focus on the last term of the inequality. By Itô's formula, we compute
\begin{align*}
    \Upsilon_{\eps, \eta} \left( M^{\hat{\Delta}^T, W}_s \right) =& \iint_{(0,s]\times \RR_+}\Upsilon'_{\eps, \eta} \left( M^{\hat{\Delta}^T, W}_{u^{-}} \right) \rho(-u) \left( \1_{\theta \leq \bar{\Lambda}^T_u}-\1_{\theta \leq X_u}\right)W(du,d\theta)\\
    &+\frac{1}{2} \int_0^s \int_{\RR_+}\Upsilon''_{\eps, \eta} \left( M^{\hat{\Delta}^T, W}_{u^{-}} \right) \rho(-u)^2 \left( \1_{\theta \leq \bar{\Lambda}^T_u}-\1_{\theta \leq X_u}\right)^2 d\theta du\\
    =& \iint_{(0,s]\times \RR_+}\Upsilon'_{\eps, \eta} \left( M^{\hat{\Delta}^T, W}_{u^{-}} \right)\rho(-u) \left( \1_{\theta \leq \bar{\Lambda}^T_u}-\1_{\theta \leq X_u}\right)W(du,d\theta) \\
    &+\frac{1}{2} \int_0^s \Upsilon''_{\eps, \eta} \left( M^{\hat{\Delta}^T, W}_{u^{-}} \right)  \rho(-u)^2\left| \bar{\Lambda}^T_u - X_u\right|du .
\end{align*}
Hence, taking the expectation, one has
\begin{equation*}
   \E{\Upsilon_{\eps, \eta} \left( M^{\hat{\Delta}^T, W}_s \right) }= \frac{1}{2} \int_0^s \E{\Upsilon''_{\eps, \eta} \left( M^{\hat{\Delta}^T, W}_{u^{-}} \right) \left| \bar{\Lambda}^T_u - X_u\right|} \rho(-u)^2 du .
\end{equation*}
By \eqref{eq: Delta_final} and the triangular inequality, we get
\begin{align*}
   \hspace*{-1cm}\E{\Upsilon_{\eps, \eta} \left( M^{\hat{\Delta}^T, W}_s \right) }\leq & C\int_0^s \E{\Upsilon''_{\eps, \eta} \left( M^{\hat{\Delta}^T, W}_{u^{-}} \right)\left|\Lambda^T_u - \bar{\Lambda}^T_u\right|} du + C\int_0^s \E{\Upsilon''_{\eps, \eta} \left( M^{\hat{\Delta}^T, W}_{u^{-}} \right)\left|M^{\hat{\Delta}^T, W}_{u^{-}} \right|} du \\
   &+C\int_0^s \E{\Upsilon''_{\eps, \eta} \left( M^{\hat{\Delta}^T, W}_{u^{-}} \right)  \left| R^T_u\right|}du .
\end{align*}
Using the properties of the Yamada's function presented in \eqref{eq: yamada}, we have
\begin{align*}
   \E{\Upsilon_{\eps, \eta} \left( M^{\hat{\Delta}^T, W}_s \right) }\leq & C\left(\frac{1}{\eta}+ \frac{e^{\eta /m^2}}{\eps \eta} \int_0^s \E{\left|R^T_u\right|} + \E{\left|\Lambda^T_u - \bar{\Lambda}^T_u \right|}du \right).
\end{align*}
By Proposition \ref{prop: discret_lambda_rho} and Proposition \ref{prop: maj_RT} with $k= \lfloor T^{4/5}\rfloor+1$, we obtain:
\begin{align*}
   \E{\Upsilon_{\eps, \eta} \left( M^{\hat{\Delta}^T, W}_s \right) }\leq & C\left(\frac{1}{\eta}+ \frac{e^{\eta /m^2}}{\eps \eta} \sqrt{\frac{1}{\sqrt{T}}+\frac{1}{\sqrt{k}} + \frac{k^2}{T^2}} \right) \leq C\left(\frac{1}{\eta}+ \frac{e^{\eta /m^2}}{\eps \eta} \frac{1}{T^{1/5}} \right).
\end{align*}
Hence, from \eqref{eq: delta_yamada} and Propositions \ref{prop: maj_RT} and \ref{prop: discret_lambda_rho}, we compute
$$\E{\left|\hat{\Delta}^T_s \right|} \leq C\left(\frac{1}{T^{1/5}}+\eps+\frac{1}{\eta}+ \frac{e^{\eta /m^2}}{\eps \eta} \frac{1}{T^{1/5}} \right).$$
We now take $\eta = \frac{m^2}{10} \ln\left(T \right)$ and $\eps = \frac{10}{\ln(T)}$ to get 
$$\E{\left|\hat{\Delta}^T_s \right|} \leq C\left(\frac{1}{T^{1/10}}+\frac{1}{\ln(T)}\right).$$

\paragraph{Upper bound of \texorpdfstring{$\E{\sup_{s\in [0,t]} \left|\Delta^T_s\right|^2}$}{}}. ~\\
We begin this part of the proof by noting that 
\begin{equation*}
    \E{\sup_{s\in [0,t]} \left| \Delta^T_s \right|^2} \leq C \left( \E{\sup_{s\in [0,t]} \left| R^T_s \right|^2} + \E{\sup_{s\in [0,t]} \left| M^{\hat{\Delta}^T,W}_{s^{-}} \right|^2}  \right).
\end{equation*}
Hence, by Proposition \ref{prop: maj_RT} with $k=\lfloor T^{4/5} \rfloor +1 $, we get
\begin{equation*}
    \E{\sup_{s\in [0,t]} \left| \Delta^T_s \right|^2} \leq C \left(\frac{1}{T^{2/5}} + \E{\sup_{s\in [0,t]} \left| M^{\hat{\Delta}^T,W}_{s^{-}} \right|^2}\right) = C \left(\frac{1}{T^{2/5}} + \E{\sup_{s\in [0,t]} \left| M^{\hat{\Delta}^T,W}_{s} \right|^2}\right).
\end{equation*}
From BDG inequality and the previous part, we conclude that
\begin{align*}
    \E{\sup_{s\in [0,t]} \left| \Delta^T_s \right|^2} &\leq C \left(\frac{1}{T^{2/5}} + \int_{0}^1 \E{\left|\hat{\Delta}^T_s\right|} ds \right) \leq C\left(\frac{1}{T^{1/10}}+\frac{1}{\ln(T)}\right).
\end{align*}
\end{proof}

We conclude this section by a corollary of Theorem \ref{th: NUHP to X} that gives the convergence rate of quantities studied in limit theorems for Hawkes processes. This corollary echoes the ideas of Theorem 2.12 in \cite{horst_functional_2024}.
\begin{cor}
\label{cor: autre maj}
    There exists $C>0$ such that for any $T>2$,
        $$(i)~~\E{\sup_{t\in [0,1]} \left| \frac{1}{T}\int_0^{tT} \lambda^{T,\natural}_s ds - \int_0^t X^\natural_s ds \right|^2} \leq \frac{C}{\ln(T)}, \quad (ii)~~\E{\sup_{t\in [0,1]} \left| \frac{H^{T,\natural}_{tT}}{T^2} - \int_0^t X^\natural_s ds \right|^2} \leq \frac{C}{\ln(T)}$$
        $$\text{and} \quad (iii)~~\E{\sup_{t\in [0,1]} \left| \frac{H^{T,\natural}_{tT}-\int_0^{tT}\lambda^{T,\natural}_s ds}{T} - \iint_{(0,t]\times \RR_+} \1_{\theta \leq X^\natural_s} W(ds,d\theta) \right|^2 }\leq \frac{C}{\ln(T)}.$$
\end{cor}
\begin{proof}
For the sake of simplicity, we drop out the index $\natural$ for all the processes.
    \begin{itemize}
        \item[(i)] It is a direct application of Theorem \ref{th: NUHP to X} with a change of variable.
        \item[(ii)] For $t\in [0,1]$, we have: 
        \begin{align*}
            \frac{H^T_{tT}}{T^2} - \int_0^t X_s ds = \frac{1}{T}\iint_{(0,t]\times \RR_+} \1_{\theta \leq \Lambda^T_s} \Ntil^T(ds,d\theta) + \int_0^t \left(\Lambda^T_s -X_s\right) ds.
        \end{align*}
        Hence, using the triangular inequality combined with Jensen's inequality, we get
        \begin{align*}
            \E{\sup_{t\in [0,1]} \left| \frac{H^T_{tT}}{T^2} - \int_0^t X_s ds \right|^2} \leq& \frac{C}{T^2}\E{\sup_{t\in [0,1]} \left| \iint_{(0,t]\times \RR_+} \1_{\theta \leq \Lambda^T_s} \Ntil^T(ds,d\theta)\right|^4}^{1/2}\\
            &+ C\E{\sup_{t\in [0,1]} \left|\Lambda^T_t -X_t\right|^2}.
        \end{align*} 
        Thus, by Theorem \ref{th: NUHP to X}, we have
        \begin{align*}
            \E{\sup_{t\in [0,1]} \left| \frac{H^T_{tT}}{T^2} - \int_0^t X_s ds \right|} \leq& \frac{C}{\ln(T)}+\frac{1}{T^2}\E{\sup_{t\in [0,1]} \left| \iint_{(0,t]\times \RR_+} \1_{\theta \leq \Lambda^T_s} \Ntil^T(ds,d\theta)\right|^4}^{1/2}.
        \end{align*}
        Moreover, by Lemma \ref{lem-outil-tilde-N-T}, we compute
        $$\E{\sup_{t\in [0,1]} \left| \iint_{(0,t]\times \RR_+} \1_{\theta \leq \Lambda^T_s} \Ntil^T(ds,d\theta)\right|^4} \leq C \E{\left|\int_0^1 \Lambda^T_s ds\right|^2} + \frac{C}{T^2}\E{\int_0^1 \Lambda^T_s ds}.$$
        This with Proposition \ref{prop: regularite_lambda} yield
        $$\E{\sup_{t\in [0,1]} \left| \iint_{(0,t]\times \RR_+} \1_{\theta \leq \Lambda^T_s} \Ntil^T(ds,d\theta)\right|^4} \leq C.$$
        Therefore,
        \begin{align*}
            \E{\sup_{t\in [0,1]} \left| \frac{H^T_{tT}}{T^2} - \int_0^t X_s ds \right|^2} \leq& \frac{C}{\ln(T)}.
        \end{align*}
    \item[(iii)] For $t\in[0,1]$, we have
    $$\hspace*{-1.5cm}\frac{H^T_{tT}-\int_0^{tT}\lambda^T_s ds}{T} - \iint_{(0,t]\times \RR_+} \1_{\theta \leq X_s} W(ds,d\theta) = \iint_{(0,t]\times \RR_+} \1_{\theta \leq \Lambda^T_s} \Ntil^T(ds,d\theta) - \iint_{(0,t]\times \RR_+} \1_{\theta \leq X_s} W(ds,d\theta).$$
    We decompose this quantity in three terms as follows:
    \begin{align*}
        \frac{H^T_{tT}-\int_0^{tT}\lambda^T_s ds}{T} - \iint_{(0,t]\times \RR_+} \1_{\theta \leq X_s} W(ds,d\theta) =& \iint_{(0,t]\times \RR_+} \1_{\theta \leq \Lambda^T_s}- \1_{\theta \leq \bar{\Lambda}^T_s} \Ntil^T(ds,d\theta)\\
        &+ \iint_{(0,t]\times \RR_+} \1_{\theta \leq \bar{\Lambda}^T_s} \left(\Ntil^T-W\right)(ds,d\theta)\\
        &+\iint_{(0,t]\times \RR_+} \1_{\theta \leq \bar{\Lambda}^T_s} - \1_{\theta \leq X_s} W(ds,d\theta)
    \end{align*}
    By the same computations that we make in the proof of Theorem \ref{th: NUHP to X}, we get
    $$\E{\sup_{t\in [0,1]} \left| \frac{H^T_{tT}-\int_0^{tT}\lambda^T_s ds}{T} - \iint_{(0,t]\times \RR_+} \1_{\theta \leq X_s} W(ds,d\theta) \right|^2 } \leq \frac{C}{\ln(T)}.$$
    \end{itemize}
\end{proof}

\section{Technical Lemmata}
\label{sec: lemmata}

This section established different results on $\Psi^{(T),\natural}$ and $\rho^\natural$ in the case of NUHP ($\natural = -$), WCHP ($\natural = 0$) and SNUHP ($\natural = +$).
We divide this section into three subsections. The first focuses on properties of $\Psi^{(T),\natural}$ and $\rho^\natural$, including the proof of Propositions \ref{prop: conv_psi_rho_-} and \ref{prop: conv_psi_rho_+} and other inequalities. The second presents key results on the different processes. In particular, we establish the assumptions required for Theorem \ref{thm: Couplage}. Finally, the last subsection is dedicated to technical elements involved in the proof of Theorem \ref{th: NUHP to X}.

\subsection{Preliminary results on \texorpdfstring{$\Psi^{(T),\natural}$ and $\rho^\natural$}{}}

We recall the notation used in the previously presented work:
$$d^{T,\natural}(u) = \Psi^{(T),\natural}(u) - \rho^\natural(u), \quad u\in \RR_+, $$
where $\Psi^{T,\natural} = \sum_{k\geq 1} \left(\phi^{T,\natural}\right)^{\ast k} = \sum_{k\geq 1} (a^\natural_T)^k \left(\phi\right)^{\ast k}$ and $\Psi^{(T),\natural}= \Psi^{T,\natural}(T \cdot)$.
Moreover, we also make use of the following notations: 
\begin{equation}
    \phi^{(T),\natural} = \phi^{T,\natural}(T \cdot)=a^\natural_T\phi(T\cdot)\quad \text{and}\quad \rho^{(T),\natural} = a^\natural_T \rho^\natural(T \cdot),
\end{equation}
where $\natural \in \{-,0,+ \}$ and $a^\natural_T$ satisfies one of the $3$ situations of interest ($a^-_T>1$, or $a^0_T=1$ or $a^+_T>1$) satisfying Assumption \ref{assump: phi}.

For the sake of completeness, we summarize the existing convergence results involving $\Psi^{(T),\natural}$ and their quantitative versions in the Table \ref{tab:psi_conv_result}. When our results fill the gaps in the literature, we state the corresponding propositions and lemmas proved in this paper, and highlight them by underlining.
\renewcommand{\arraystretch}{1.8}
\begin{table}[!ht]
    \centering
    \hspace*{-1cm}\begin{tabular}{|C{1.5cm}|C{7.5cm}|C{6.5cm}|}
    \cline{2-3} \multicolumn{1}{c|}{}  & $\left\| \Psi^{(T), \natural} - \rho^\natural\right\|_{\LL^2(0,1)} \xrightarrow[T \to +\infty]{} 0$ & $\left\|\Psi^{(T), \natural} - \rho^\natural \right\|_{\LL^2(0,1)}  \leq \frac{C}{\sqrt{T}}$  \\ \hline
        NUHP \newline $(\natural = -)$ & Lemma 4.5, \cite{jaisson_limit_2015} & \underline{\textit{Proposition \ref{prop: conv_psi_rho_-}}} \\ \hline
        WCHP \newline $(\natural = 0)$ & Theorem 2.12, \cite{horst_functional_2024} & Proof of Proposition 2.13, \cite{horst_functional_2024} \\ \hline
        SNUHP \newline $(\natural = +)$& $\bullet$ Pointwise convergence: Lemma 2.2, \cite{liu_scaling_2024} \newline $\bullet$ \underline{\textit{$\LL^2(0,1)$ convergence: Proposition \ref{prop: conv_psi_rho_+}}} & \underline{\textit{Proposition \ref{prop: conv_psi_rho_+}}} \\ \hline
    \end{tabular}
    \caption{State of the art of convergence results}
    \label{tab:psi_conv_result}
\end{table}

Whereas the proof of Proposition \ref{prop: conv_psi_rho_-} and \ref{prop: conv_psi_rho_+} are similar, more work is necessary to deal with the case $\left\| \phi^{T,+}\right\| >1$. However, it is important to note that the methodology of both proofs is inspired by the methodology used in \cite{xu_diffusion_2024} or \cite{xu_scaling_2024}. Indeed, the authors of \cite{xu_scaling_2024} prove the convergence given in Proposition \ref{prop: conv_psi_rho_0} which deals with WCHP. Here, for the sake of completeness, we decide to adapt their proof for the two other cases. To do so, we make use of the Fourier transform defined as follows:
$$\Fcal(f): \xi \mapsto \int_{\RR} e^{\i \xi t} f(t) dt.$$

\subsubsection{Proof of Proposition \ref{prop: conv_psi_rho_-}}
\label{subs: proof conv_psi_rho_-}

In this part, we established the convergence rate presented in Proposition \ref{prop: conv_psi_rho_-}.
During the proof we will use the $\LL^2$-norm and so we will need the following lemma:
\begin{lem}
\label{lem: Psi_L2}
For any $T>1$, $\Psi^{(T),-} \in \LL^2(\RR_+)$.
\end{lem}
Thanks to this lemma, we will be able to make use of Fourier analysis, particularly the Fourier-Plancherel equality. To prepare the proofs, for $z\in \RR$, we give the expressions of $\Fcal \psi^{(T),-}(z)$ and $\Fcal \rho(z)$:
$$\Fcal \Psi^{(T),-}(z) = \frac{\Fcal \phi^{(T),-}(z) }{1-T\Fcal \phi^{(T),-}(z)} =\frac{1}{T\left(1-T\Fcal \phi^{(T),-}(z)\right)}-\frac{1}{T} \quad \text{and} \quad \Fcal \rho^{-} (z) = \frac{1}{1-\i m z}.$$
\begin{proof}[Proof of Lemma \ref{lem: Psi_L2}]
    Let $T>1$. Using the fact that $\Psi^{T,-} = \phi^{T,-} + \Psi^{T,-} * \phi^{T,-}$ and $\Psi^{(T),-} = \Psi^{T,-}(T\cdot)$ and $\phi \in \LL^1(\RR_+)$, we deduce that $\Psi^{T,-} \in \LL^1(\RR_+)$ is well defined (see Theorem 1.1 of \cite{grossman_notes_1979}) and one can prove that $\Psi^{T,-}, \Psi^{(T),-} \in \LL^1(\RR_+)$ with
    $$\left\| \Psi^{T,-}\right\|_1 = \frac{a^{-}_T}{1-a^{-}_T} \quad \text{and} \quad  \left\| \Psi^{(T),-} \right\|_1 = a^{-}_T.$$
    Moreover, we also have:
    $$\left\| \Psi^{T,-} \right\|_{\infty} \leq \left\| \phi^{T,-} \right\|_{\infty} \left(1 + \left\| \Psi^{T,-}\right\|_{1} \right)< \infty .$$
    Hence, we conclude that $\Psi^{(T),-} \in \LL^1(\RR_+) \cap \LL^{\infty}(\RR_+) \subset \LL^2(\RR_+) $.
\end{proof}

In order to prove Proposition \ref{prop: conv_psi_rho_-}, we introduce Lemma \ref{lem: maj_eps_3_-} that allows us to control a key quantity for our quantification:
$$\eps^T(z) := \Fcal d^{T,-} (z)+ \frac{1}{T} =\frac{\i z m -1 + T\left(1- T\Fcal\phi^{(T),-}(z) \right) }{T\left(1- T\Fcal\phi^{(T),-}(z) \right)(\i zm -1)}, \quad z\in \RR_+ .$$

\begin{lem}
    \label{lem: maj_eps_3_-}
    There exist $C>0$ and $\alpha>0$ such that for any $T\geq 2$ and $|z|\leq \alpha T$, 
    $$\left|\eps^T(z) \right| \leq \frac{4}{T \sqrt{|z|^2m^2+1}} + \frac{ 4 |z| m_2}{T m \sqrt{|z|^2m^2+1}}.$$
\end{lem}

\begin{proof}
    Let us recall that $\int_{\RR_+} \phi(t) dt=1$, $m=\int_{\RR_+} t\phi(t) dt$ and $a^{-}_T= 1-\frac{1}{T}$. Then, for any real $z$, we have:
    \begin{align*}
      -\eps^T(z) &= \frac{T\int_0^{+\infty} \left(a^{-}_T e^{\frac{\i z}{T}t}-a^{-}_T -\frac{\i z}{T}t\right)\phi(t) dt}{T \left( 1-T\Fcal \phi^{(T),-}(z)\right)\left(\i zm-1\right)}\\
        &=\frac{-\int_0^{+\infty} \left( e^{\frac{\i z}{T}t}-1\right)\phi(t) dt}{T \left( 1-T\Fcal \phi^{(T),-}(z)\right)\left(\i zm-1\right)} + \frac{T\int_0^{+\infty} \left( e^{\frac{\i z}{T}t}-1-\frac{\i z}{T}t\right)\phi(t) dt}{T \left( 1-T\Fcal \phi^{(T),-}(z)\right)\left(\i zm-1\right)}.
    \end{align*}
    Using the fact that for any $\xi \in \i \RR$, $\left|e^\xi -1 \right| \leq \left|\xi\right|$ and $\left|e^\xi -1-\xi \right| \leq \left|\xi\right|^2 $ and $m_2=\int_{\RR_+} t^2\phi(t) dt < +\infty.$ we get:
    \begin{align*}
         \left| -\int_0^{+\infty} \left( e^{\frac{\i z}{T}t}-1\right)\phi(t) dt \right| &\leq \frac{|z|}{T}\int_0^{+\infty} t \phi(t) dt = \frac{|z|m}{T}\\
        \text{and  } \left| T\int_0^{+\infty} \left( e^{\frac{\i z}{T}t}-1-\frac{\i z}{T}t\right)\phi(t) dt\right| & \leq \frac{|z|^2}{T}\int_0^{+\infty} t^2 \phi(t) dt = \frac{|z|^2 m_2}{T}.
    \end{align*}
    Hence, 
    \begin{align*}
        \left| T\int_0^{+\infty} \left(a_T e^{\frac{\i z}{T}t}-a^{-}_T -\frac{\i z}{T}t\right)\phi(t) dt \right| &\leq \frac{|z|}{T}\int_0^{+\infty} t \phi(t) dt + \frac{|z|^2}{T}\int_0^{+\infty} t^2 \phi(t) dt.
    \end{align*}
    Moreover, according to \cite{jaisson_limit_2015} (Lemma 4.4), there exists $\alpha>0$ such that for any $|x|\leq \alpha$,
    $$\left|\text{Im} \Fcal \phi (x)\right| \geq \frac{m|x|}{2}.$$
    Thus, for any $|z|\leq \alpha T$, we have:
    $$\left| 1- T \Fcal \phi^{(T),-}(z)\right| = \left| 1 - a^{-}_T \Fcal \phi \left(\frac{z}{T} \right)\right| \geq a_T\left| \text{Im} \Fcal \phi \left(\frac{z}{T} \right)\right| \geq a^{-}_T \frac{m|z|}{2T}$$
    and so, using the fact that $a^{-}_T \geq \frac{1}{2}$ for $T\geq 2$, we obtain
    $$T\left| 1- T \Fcal \phi^{(T),-}(z)\right| \geq a^{-}_T\frac{m|z|}{2} \geq \frac{m|z|}{4}.$$
    Therefore, for any $|z|\leq \alpha T$,
    \begin{align*}
        \left| \eps^T(z)\right| \leq \frac{4}{T |\i z m-1|} + \frac{ 4 |z| m_2}{T m |\i z m-1|} \leq \frac{4}{T \sqrt{|z|^2m^2+1}} + \frac{ 4 |z| m_2}{T m \sqrt{|z|^2m^2+1}}
    \end{align*}
\end{proof}

We now give the proof of Proposition \ref{prop: conv_psi_rho_-}.

\begin{proof}[Proof of Proposition \ref{prop: conv_psi_rho_-}]
    We fix $T\geq 2$. Note that since $\Psi^{(T),-}\in \LL^2(\RR_+)$ (see Lemma \ref{lem: Psi_L2}), we have
    $$\left\|\Psi^{(T),-}-\rho^{-} \right\|_{\LL^2(0,1)} \leq \left\|\Psi^{(T),-}-\rho^{-} \right\|_2 = 2\pi\left\|\Fcal\Psi^{(T),-}-\Fcal\rho^{-} \right\|_2.$$
    According to Lemma \ref{lem: maj_eps_3_-}, there exists $\alpha>0$ such that for any $|z|\leq \alpha T$,
    $$\left|\eps^T(z) \right| \leq \frac{4}{T \sqrt{|z|^2 m^2+1}} + \frac{ 4 |z| m_2}{T m \sqrt{|z|^2m^2+1}}$$
    where $\eps^T$ is the same as Lemma \ref{lem: maj_eps_3_-}. Moreover, we make use of the following notations:
    $$\eps_1^{T} := \int_{|z|\geq \alpha T} \left| \Fcal \Psi^{(T),-}(z) - \Fcal \rho^{-}(z)\right|^2 dz \quad \text{and} \quad \eps_2^T :=\int_{|z|<\alpha T} \left| \Fcal \Psi^{(T),-}(z) - \Fcal \rho^{-}(z)\right|^2 dz.$$
    Let us focus on $\eps_1^T$. Using Corollary 4.3 in \cite{jaisson_limit_2015}, we have:
    $$\eps_1^{T} \leq C \int_{|z|\geq \alpha T} \left(\frac{1}{|z|}\wedge 1 \right)^2 dz \leq \frac{C}{T}.$$
    We now deal with $\eps_2^T$. This part of the proof relies on the upper bound of $\eps^T$ proved in Lemma \ref{lem: maj_eps_3_-}. First, note that
    $$\Fcal \Psi^{(T),-} = \frac{1}{T\left(1- T\Fcal\phi^{(T),-} \right)}-\frac{1}{T}.$$
    So, for $z\in \RR$, we have
    $$\left| \Fcal \Psi^{(T),-}(z) - \Fcal \rho^{-}(z) \right| \leq \frac{1}{T} + \left| \eps^T(z) \right|.$$
    Hence, for any $0<|z|\leq \alpha T$, we get
    \begin{align*}
        \left| \Fcal \Psi^{(T),-}(z) - \Fcal \rho^{-}(z) \right|^2 &\leq \frac{C}{T^2} + C\left| \eps^T(z) \right|^2 \leq \frac{C}{T^2} + \frac{C}{T^2 (|z|^2m^2+1)} + \frac{C |z|^2 m_2^2}{T^2 m^2 (|z|^2m^2+1)}.
    \end{align*}
    Thus, we get
    \begin{align*}
        \left| \eps_2^T \right| &\leq \int_{|z|<\alpha T} \frac{C}{T^2} + \frac{C}{T^2 (|z|^2m^2+1)} + \frac{C |z|^2 m_2^2}{T^2 m^2 (|z|^2m^2+1)} dz \leq \frac{C}{T}.
    \end{align*}
\end{proof}

\begin{remark}
One can observe that we refer to results presented in \cite{jaisson_limit_2015} without assuming $\sup_{T>1} \|\Psi^{(T),-}\|_{\infty} < +\infty$, which is assumed to hold in \cite{jaisson_limit_2015}. However, the proofs of Lemma 4.4 and Corollary 4.3 in \cite{jaisson_limit_2015} do not rely on this supremum assumption. Therefore, we can apply these results without requiring this condition to hold extending the convergence result.
\end{remark}

\subsubsection{Proof of Proposition \ref{prop: conv_psi_rho_+}}
\label{subs: proof conv_psi_rho_+}
Inspired by the work in \cite{liu_scaling_2024}, we will prove Proposition \ref{prop: conv_psi_rho_+} by using Malthusian parameter. To do so, we recall some element presented in \cite{liu_scaling_2024} including the definition of such parameter. For $T\geq 2$ Malthusian parameter $b_T>0$ are defined in the literature by the following relation:
$$\int_0^\infty e^{-b_T s} \phi^{T,+}(s) ds = \frac{1}{a^{+}_{T}}<1.$$
Thus, we will proved the result for the case $a_T >1$ by adapting the proof of Proposition \ref{prop: conv_psi_rho_-} (which deals with the case $a_T <1$) and by using the Malthusian parameter. We thereupon set for $t\in \RR_+$,
$$\tilde{\phi}^{T,-}(t) := e^{-b_T t} \phi^{T,+}(t) \quad \text{and} \quad \tilde{\Psi}^{(T),-} := e^{-b_T T t} \Psi^{(T),+}(t) .$$
In particular, one can notice that for any $k\geq 1$,
$$ e^{-b_Tt} (\phi^{T,+})^{*k}(t)= (\tilde{\phi}^{T,-})^{*k}(t) \quad \text{and} \quad  \tilde{\Psi}^{(T),-}(t)=\sum_{k=1}^{\infty} (\tilde{\phi}^{T,-})^{*k}(Tt).$$
Note that convergence of the serie is due to $\left\|\tilde{\phi}^{T,-}\right\|_{L^1(\RR^+)}=\frac{1}{a_T^+} < 1.$\\
For the sake of completeness, we also define $\tilde{m}_{T}$ as the first moment of $\tilde{\phi}^{T,-}$ that is
$$ \tilde{m}_{T} :=\int_0^\infty s  \tilde{\phi}^{T,-}(s) ds =\int_0^\infty s e^{-b_T s} \phi^{T,+}(s) ds.$$
We now recall the different convergence phenomenons that have been proved in \cite{liu_scaling_2024}. Indeed, Proof of Lemma 2.2 in \cite{liu_scaling_2024} claims that
\begin{equation}
\label{eq: cvg_malthusian_param}
    \lim_{T\to +\infty} b_T = 0,\quad \lim_{T\to +\infty} Tb_T = \frac{2}{m} \quad \text{and} \quad \lim_{T\to +\infty} \tilde{m}_T = m.
\end{equation}

With these definitions, we are able to bound the $\LL^2(0,1)$ norm of $\Psi^{(T),+}- \rho^{+}$ above by the  $\LL^2(0,1)$ norm of $\left\| \tilde{\Psi}^{(T),-}- \tilde{\rho}^{T, -}\right\|_{\LL^2(0,1)}$ where $\tilde{\rho}^{T, -}(t) := e^{-b_T T t} \rho^{+}(t) = \frac{1}{m}e^{-t\left(b_T T -\frac{1}{m} \right)}$. That is
\begin{align}
    \label{eq: ineq_Psi_+}
    \left\| \Psi^{(T),+}- \rho^{+}\right\|_{\LL^2(0,1)} \leq e^{b_T T} \left\| \tilde{\Psi}^{(T),-}- \tilde{\rho}^{T, -}\right\|_{\LL^2(0,1)}.
\end{align}

This allows us to realize computations using the Fourier transform of $\tilde{\Psi}^{(T),-}$. In particular, we get the following lemma:
\begin{lem}
\label{lem: Psi_L2_2}
For any $T>1$, $\tilde{\Psi}^{(T),-} \in \LL^2(\RR_+)$ and
$$\Fcal \tilde{\Psi}^{(T),-}(z) = \frac{\Fcal \tilde{\phi}^{(T),-}(z) }{1-T\Fcal \tilde{\phi}^{(T),-}(z)}, \quad z\in \RR.$$
\end{lem}
\begin{proof}
    Let $T>1$. Since $\tilde{\phi}^{T,-} \in \LL^1(\RR_+)$ and since for any $\xi \geq 0$, $$\int_0^{\infty} e^{-\xi s}\tilde{\phi}^{T,-}(s) ds \leq \left\| \tilde{\phi}^{T,-} \right\|_{1} = \frac{1}{a^{+}_T} <1 ,$$ 
    Theorem 1.1 of \cite{grossman_notes_1979} yields $\tilde{\Psi}^{T,-}\in \LL^1(\RR_+)$ and so $\tilde{\Psi}^{(T),-}\in \LL^1(\RR_+)$.
    Moreover, we also have:
    \begin{equation}
    \label{eq: maj_infty_psi_tilde_T}
        \left\| \tilde{\Psi}^{T,-} \right\|_{\infty} \leq \left\| \tilde{\phi}^{T,-} \right\|_{\infty} \left(1 + \left\| \tilde{\Psi}^{T,-}\right\|_{1} \right)< +\infty .
    \end{equation}
    Hence, we conclude that $\tilde{\Psi}^{(T),-} \in \LL^1(\RR_+) \cap \LL^{\infty}(\RR_+) \subset \LL^2(\RR_+) $. 
    On the other hand, the Fourier transform of $\tilde{\Psi}^{(T),-}$ can be deduce using the equality
    $$\tilde{\Psi}^{(T),-} = \tilde{\phi}^{(T),-} + T \tilde{\phi}^{(T),-} \ast \tilde{\Psi}^{(T),-}.$$
\end{proof}
In order to give the convergence rate of our $d^{T,+} = \Psi^{(T),+}-\rho^{+}$, we need to focus on the Fourier transform of $\tilde{\Psi}^{(T),-}$ and $\tilde{\rho}^{T,-}$. The following lemma is similar to Proposition 5.1 in \cite{xu_diffusion_2024} which can be seen as an extension of Lemma 4.4 in \cite{jaisson_limit_2015}.
\begin{lem}
\label{lem: maj_min_tsfo_Psi_+}
    There exists $C>0$ such that for any $T>1$,
    $$\left| \Fcal \tilde{\phi}^{T,-} (z) \right| \leq C \left(\frac{1}{|z|}\wedge 1 \right), \quad \forall z\in \RR.$$
    Moreover, there exist $C>0$ and $T_0 >1$ such that for any $T>T_0$,
    $$\left|1- \Fcal\tilde{\phi}^{T,-} (z) \right| \geq C \left( |z| \wedge 1\right), \quad \forall z \in \RR.$$
\end{lem}
\begin{proof}
    The first inequality is exactly the result obtain in Proposition 5.1 of \cite{xu_diffusion_2024}. In particular, since the total variation of $\tilde{\psi}^{T,-}$ is uniformly bounded with respect to $T$, we get that there exists $C>0$ such that for any $z\in \RR$ and any $T>1$,
    $$\left| \Fcal \tilde{\phi}^{T,-} (z) \right| \leq C \left(\frac{1}{|z|}\wedge 1 \right).$$
    The second inequality require more computations.\\
    According to the proof of Lemma 4.4 in \cite{jaisson_limit_2015}, there exists $\alpha >0$ and $\eps>0$ such that if $|z| \leq \alpha$,
    $$\left| \text{Im}\left( \Fcal \phi \right)(z) \right| \geq \frac{m}{2}|z|$$
    and if $|z| \geq \alpha$,
    $$\left| 1- \Fcal \phi (z) \right| \geq \epsilon .$$
    We first deal with the case $|z| \leq \alpha$. Notice that since $a^{+}_T>1$, we have
    \begin{align*}
        \left|1- \Fcal\tilde{\phi}^{T,-} (z) \right| &\geq \left|\text{Im}\left( \int_0^\infty e^{itz}e^{-b_T t} \phi^{T,+}(t) dt \right) \right| = \left| \int_0^\infty e^{-b_T t} \sin(zt) \phi(t) dt \right|.
    \end{align*}
    By triangular inequality, we get
    $$\left|1- \Fcal\tilde{\phi}^{T,-} (z) \right| \geq \left|\text{Im}\left( \Fcal \phi \right)(z) \right| - \left| \int_0^\infty \left(1-e^{-b_T t}\right) \sin(zt) \phi(t) dt \right|.$$
    Moreover, since $|1- e^{-x}|\leq x$ for any $x\geq 0$ and $\left| \sin(u)\right| \leq |u|$, we obtain that
    $$\left| \int_0^\infty \left(1-e^{-b_T t} \right) \sin(zt) \phi(t) dt\right| \leq |z| b_T \int_0^\infty t^2 \phi(t) dt = |z| m_2 b_T .$$
    In addition, using the fact that $\lim_{T\to \infty} b_T =0 $, there exists $T_0>1$ such that for any $T>T_0$, $b_T m_2 \leq \frac{m}{4}$. Hence, for any $T>T_0$,
    \begin{equation}
    \label{eq: maj_denom_tsfo_psi_1}
    \left|1- \Fcal\tilde{\phi}^{T,-} (z) \right| \geq \frac{m}{2}|z| - \frac{m}{4}|z| = \frac{m}{4}|z|, \quad \forall |z|\leq \alpha .
    \end{equation}
    We now consider the case $|z| \geq \alpha$ by noticing that
    $$\left|1- \Fcal\tilde{\phi}^{T,-} (z) \right| \geq \left|1- \Fcal\phi(z) \right| - \left|\int_0^\infty \left(a^{+}_T e^{-b_T t} -1 \right) \phi(t) e^{izt} dt\right|.$$
    Moreover, thanks to $a^{+}_T = 1+ \frac{1}{T}$, we also have
    $$\left| a^{+}_T e^{-b_T t} -1 \right| \leq \left| a^{+}_T -1 \right| + a^{+}_T \left|e^{-b_T t} - 1\right| \leq \frac{1}{T} + a^{+}_T b_T t.$$
    We choose $T_0>1$ large enough in order to get for any $T>T_0$
    $$\frac{1}{T} + a^{+}_T b_T m \leq \frac{\eps}{2}.$$
    Thereupon, we obtain for any $T>T_0$
    $$\left|\int_0^\infty \left(a^{+}_T e^{-b_T t} -1 \right) \phi(t) e^{izt} dt\right| \leq \int_0^\infty \left( \frac{1}{T} + a^{+}_T b_T t \right) \phi(t) dt = \frac{1}{T} + a^{+}_T b_T m \leq \frac{\eps}{2}. $$
    Thus, for any $T>T_0$, we compute
    \begin{equation}
    \label{eq: maj_denom_tsfo_psi_2}
        \left|1- \Fcal\tilde{\phi}^{T,-} (z) \right| \geq \eps - \frac{\eps}{2} = \frac{\eps}{2}, \quad |z| \geq \alpha.
    \end{equation}
    Therefore, putting together equalities \eqref{eq: maj_denom_tsfo_psi_1} and \eqref{eq: maj_denom_tsfo_psi_2}, we conclude that there exist $C>0$ and $T_0>1$ such that for any $T>T_0$, and for any $z \in \RR$,
    $$\left|1- \Fcal\tilde{\phi}^{T,-} (z) \right| \geq C \left( |z| \wedge 1\right).$$
\end{proof}

A last technical result is required in order to prove Proposition \ref{prop: conv_psi_rho_+}. The latter gives an upper bound of $\tilde{\eps}^T$ which is defined by
$$\tilde{\eps}^T(z) := \left[\Fcal \tilde{\Psi}^{(T),-} (z) - \Fcal \tilde{\rho}^{T,-} (z) \right]+ \frac{1}{T}$$
with $\tilde{\Psi}^{(T),-}(t) = e^{-b_T T t} \Psi^{(T),+}(t)$ and $\tilde{\rho}^{T,-} (t) = e^{-b_T T t} \rho^{+}(t)$.
\begin{lem}
    \label{lem: maj_eps_3_+}
    For any $\alpha \in \left( 0, \frac{m}{8m_2}\right)$, there exist $C>0$ and $T_0>4$ such that for any $T\geq T_0$ and $|z|\leq \alpha T$, 
    $$\left|\tilde{\eps}^T(z) \right| \leq \frac{C}{T}.$$
\end{lem}
\begin{proof}
    First of all, note that in this proof, $T_0$ may change from line to line. We now start the proof. Since for any $z\in \RR$,
    $$\hspace*{-0.5cm}\Fcal \tilde{\Psi}^{(T),-}(z) = \frac{\Fcal \tilde{\phi}^{(T),-}(z) }{1-T\Fcal \tilde{\phi}^{(T),-}(z)} =\frac{1}{T\left(1-T\Fcal \tilde{\phi}^{(T),-}(z)\right)}-\frac{1}{T} \quad \text{and} \quad \Fcal \tilde{\rho}^{T,-} (z) = \frac{1}{m b_T T -1 -\i m z},$$ 
    we compute
    $$\tilde{\eps}^T(z) =\frac{\i z m + 1-b_T T m + T\left(1- T\Fcal\tilde{\phi}^{(T),-}(z) \right) }{T\left(1- T\Fcal\tilde{\phi}^{(T),-}(z) \right)(\i z m + 1-b_T T m)}, \quad z\in \RR .$$
    Using that and $a^{+}_T= 1+\frac{1}{T}$ (and recalling that $\left\| \phi \right\|_1=1$), we get
    \begin{align*}
      -\tilde{\eps}^T(z) &= \frac{T\int_0^{+\infty} \left(a^{+}_T e^{\left(\frac{\i z}{T}-b_T\right)t}-a^{+}_T -\left(\frac{\i z}{T}-b_T\right)t \right)\phi(t) dt}{T \left( 1-T\Fcal \tilde{\phi}^{(T),-}(z)\right)\left(\i z m + 1-b_T T m\right)}\\
        &=\frac{\int_0^{+\infty} \left( e^{\left(\frac{\i z}{T}-b_T\right)t}-1\right)\phi(t) dt}{T \left( 1-T\Fcal \tilde{\phi}^{(T),-}(z)\right)\left(\i z m + 1-b_T T m\right)} + \frac{T\int_0^{+\infty} \left( e^{\left(\frac{\i z}{T}-b_T\right)t}-1-\left(\frac{\i z}{T}-b_T\right)t\right)\phi(t) dt}{T \left( 1-T\Fcal \tilde{\phi}^{(T),-}(z)\right)\left(\i z m + 1-b_T T m\right)}.
    \end{align*}
    Moreover, it holds that for any $\xi \in \CC_{-} := \left\{ x+\i y \in \CC \mid x\leq 0\right\}$, $\left|e^\xi -1 \right| \leq \left|\xi\right|$ and $\left|e^\xi -1-\xi \right| \leq \left|\xi\right|^2 $. Thus, we get:
    \begin{align*}
         \left| -\int_0^{+\infty} \left( e^{\left(\frac{\i z}{T}-b_T\right)t}-1\right)\phi(t) dt \right| &\leq \frac{|\i z-b_T T|}{T}\int_0^{+\infty} t \phi(t) dt \leq \frac{\left(|z|+b_T T\right)m}{T},\\
        \left| T\int_0^{+\infty} \left( e^{\left(\frac{\i z}{T}-b_T\right)t}-1-\left(\frac{\i z}{T}-b_T\right)t\right)\phi(t) dt\right| & \leq \frac{|\i z-b_T T|^2}{T}\int_0^{+\infty} t^2 \phi(t) dt = \frac{\left(|z|^2+|b_T T|^2 \right)m_2}{T}.
    \end{align*}
    Hence, 
    \begin{align*}
        \left| T\int_0^{+\infty} \left(a^{+}_T e^{\left(\frac{\i z}{T}-b_T\right)t}-a^{+}_T -\left(\frac{\i z}{T}-b_T\right)t\right)\phi(t) dt \right| &\leq \frac{\left(|z|+b_T T\right)m}{T} + \frac{\left(|z|^2+|b_T T|^2 \right)m_2}{T}.
    \end{align*}
We deduce that for any $z\in \RR$ and $T>1$,
\begin{align*}
    \left|T \left(1- T\Fcal \tilde{\phi}^{(T),-}\left(z \right)\right) \right| & = \left|T\int_0^{+\infty} \left(a^{+}_T e^{\frac{\i z}{T}t}-a^{+}_T -\frac{\i z}{T}t\right)\phi(t) dt  - \left(\i z m + 1-b_T T m\right)\right|\\
    &\geq \left|\i z m + 1-b_T T m\right| - \left|T\int_0^{+\infty} \left(a^{+}_T e^{\frac{\i z}{T}t}-a^{+}_T -\frac{\i z}{T}t\right)\phi(t) dt\right|\\
    &\geq \frac{m}{2}|z| + \frac{\sqrt{2}}{2}\left| m b_T T - 1 \right| - \frac{\left(|z|+b_T T\right)m}{T} - \frac{\left(|z|^2+|b_T T|^2 \right)m_2}{T}.
\end{align*}
Here, we have employed the inequality 
$$|m \i z + b_T T m -1| \geq \frac{m}{2}|z| + \frac{\sqrt{2}}{2}\left| m b_T T - 1 \right|.$$ 
Thus, there is some $T_0>4$ such that for any $T>T_0$ and $|z|\leq \alpha T$, 
\begin{align*}
    \left|T \left(1- T\Fcal \tilde{\phi}^{(T),-}\left(z \right)\right) \right| &\geq |z|\left(\frac{m}{2}-\frac{m}{T}-\alpha m_2 \right) + \frac{\sqrt{2}}{2}\left| m b_T T - 1 \right| - b_T m - \frac{|b_T T|^2 m_2}{T}\\
    &\geq |z|\left(\frac{m}{4}-\alpha m_2 \right) + \frac{\sqrt{2}}{2}\left| m b_T T - 1 \right| - b_T m - \frac{|b_T T|^2 m_2}{T}.
\end{align*}
Since $\alpha \in \left( 0, \frac{m}{8m_2}\right)$, we get for any $T>T_0$ and $|z|\leq \alpha T$,
\begin{align*}
    \left|T \left(1- T\Fcal \tilde{\phi}^{(T),-}\left(z \right)\right) \right| &\geq |z|\frac{m}{8} + \frac{\sqrt{2}}{2}\left| m b_T T - 1 \right| - b_T m - \frac{|b_T T|^2 m_2}{T}.
\end{align*}
In addition, we know that $\lim_{T\to +\infty} b_T + \frac{|b_T T|^2 m_2}{T} = 0$ which, combined with $\lim_{T\to +\infty} \left|1-b_T T m \right| = 1$, gives the existence of some $T_0>4$ such that for any $T>T_0$,
$$b_T m + \frac{|b_T T|^2 m_2}{T} \leq \left( \frac{\sqrt{2}}{2}-\frac{1}{8} \right)\left|1-b_T T m \right|.$$
This implies that for any $T>T_0$ and $|z| \leq \alpha T$,
\begin{align*}
    \left|T \left(1- T\Fcal \tilde{\phi}^{(T),-}\left(z \right)\right) \right| &\geq \frac{m}{8}|z|+ \frac{1}{8}\left| m b_T T - 1 \right|.
\end{align*}
Therefore, for any $T>T_0$ and $|z|\leq \alpha T$,
\begin{equation}
    \label{eq: maj_eps_tilde}
    \left|\tilde{\eps}^{T}(z)\right| \leq \frac{8\left(|z|+b_T T\right)m}{T \left(m|z|+ \left| m b_T T - 1 \right|\right)\left|\i z m + 1-b_T T m\right|} + \frac{8\left(|z|^2+|b_T T|^2 \right)m_2}{T \left(m|z|+ \left| m b_T T - 1 \right|\right)\left|\i z m + 1-b_T T m\right|}.
\end{equation}
Let us focus on the first term on the right hand side of the equation. In particular, since $\lim_{T\to +\infty} m b_T T = 2$, there exists $T_0>4$ such that for any $T>T_0$,
$$m b_T T \leq \frac{5}{2} \quad \text{and} \quad \left|mb_T T -1\right| \geq \frac{1}{2}.$$
This gives that for any $T>T_0$ and $|z|\leq \alpha T$,
$$\frac{8\left(|z|+b_T T\right)m}{T \left(m|z|+ \left| m b_T T - 1 \right|\right)\left|\i z m + 1-b_T T m\right|} \leq \frac{8\left(|z|m+5/2\right)}{T \left(m|z|+ 1/2 \right)\left|\i z m + 1-b_T T m\right|}.$$
And so using the inequality
$$|m \i z + b_T T m -1| \geq \frac{m}{2}|z| + \frac{1}{2}\left| m b_T T - 1 \right| \geq \frac{m}{2}|z| + \frac{1}{4},$$
we obtain
\begin{equation*}
    \frac{8\left(|z|+b_T T\right)m}{T \left(m|z|+ \left| m b_T T - 1 \right|\right)\left|\i z m + 1-b_T T m\right|} \leq \frac{16\left(|z|m+5/2\right)}{T \left(m|z|+ 1/2 \right)\left(m|z|+ 1/2 \right)}.
\end{equation*}
Hence, for any $T>T_0$ and $|z|\leq \alpha T$, we get
\begin{equation}
\label{eq: maj_eps_tilde_1}
    \frac{8\left(|z|+b_T T\right)m}{T \left(m|z|+ \left| m b_T T - 1 \right|\right)\left|\i z m + 1-b_T T m\right|} \leq \frac{C}{T}.
\end{equation}
Similarly, for any $T>T_0$ and $|z|\leq \alpha T$, we get
\begin{equation}
\label{eq: maj_eps_tilde_2}
    \frac{8\left(|z|^2+|b_T T|^2 \right)m_2}{T \left(m|z|+ \left| m b_T T - 1 \right|\right)\left|\i z m + 1-b_T T m\right|} \leq \frac{16\left(m^2|z|^2+25/4\right)m_2}{T m^2\left(m|z|+ 1/2\right)^2} \leq \frac{C}{T}
\end{equation}
Combining \eqref{eq: maj_eps_tilde_1} and \eqref{eq: maj_eps_tilde_2} in \eqref{eq: maj_eps_tilde}, we obtain that there exist $C>0$ and $T_0>4$ such that for any $T>T_0$ and $|z|\leq \alpha T$, 
$$\left| \tilde{\eps}^T (z) \right| \leq \frac{C}{T}.$$
\end{proof}

\begin{proof}[Proof of Proposition \ref{prop: conv_psi_rho_+}]
    We first recall to the reader inequality \eqref{eq: ineq_Psi_+} which is the starting point of our proof:
    $$\left\| \Psi^{(T),+} - \rho^{+}\right\|_{\LL^2(0,1)} \leq e^{b_T T} \left\| \tilde{\Psi}^{(T),-} - \tilde{\rho}^{T,-}\right\|_{\LL^2(0,1)}.$$
    Thanks to Lemma \ref{lem: Psi_L2_2}, we are able to make use of Fourier Plancherel equality which yields
    $$\left\| \tilde{\Psi}^{(T),-} - \tilde{\rho}^{T,-}\right\|_{\LL^2(0,1)} \leq \left\| \tilde{\Psi}^{(T),-} - \tilde{\rho}^{T,-}\right\|_{2} = 2\pi \left\| \Fcal\tilde{\Psi}^{(T),-} - \Fcal\tilde{\rho}^{T,-}\right\|_{2}.$$
    Thanks to Lemma \ref{lem: maj_eps_3_+}, for some $\alpha \in \left( 0, \frac{m}{8m_2}\right)$ there exist $C>0$ and $T_0>4$ such that for any $T>T_0$ and $|z|\leq \alpha T$,
    $$\left| \tilde{\Psi}^{(T),-}(z) - \tilde{\rho}^{T,-}(z) \right| \leq \left|\tilde{\eps}^T(z) \right|+\frac{1}{T} \leq \frac{C}{T}.$$
    Hence,
    \begin{align*}
        \hspace*{-0.5cm}\left\| \Fcal\tilde{\Psi}^{(T),-} - \Fcal\tilde{\rho}^{T,-}\right\|_{2}^2 &= \int_{|z|\leq \alpha T} \left| \Fcal\tilde{\Psi}^{(T),+}(z) - \Fcal\tilde{\rho}^{T,+}(z) \right|^2 dz +  \int_{|z|> \alpha T} \left| \Fcal\tilde{\Psi}^{(T),+}(z) - \Fcal\tilde{\rho}^{T,+}(z) \right|^2 dz \\
        &\leq \frac{C}{T} + \int_{|z|> \alpha T} \left| \Fcal\tilde{\Psi}^{(T),+}(z) - \Fcal\tilde{\rho}^{T,+}(z) \right|^2 dz.
    \end{align*}
    On the other hand, thanks to Lemma \ref{lem: maj_min_tsfo_Psi_+}, it can be proved that there exist $C>0$ and $T_0>4$ such that for any $T>T_0$,
    $$\left| \Fcal\tilde{\Psi}^{(T),-}(z) - \Fcal\tilde{\rho}^{T,-}(z) \right| \leq C \left(\frac{1}{|z|}\wedge 1 \right), \quad \forall z\in \RR.$$
    In particular, this inequality yields, for $T>T_0$,
    $$\int_{|z|> \alpha T} \left| \Fcal\tilde{\Psi}^{(T),-}(z) - \Fcal\tilde{\rho}^{T,-}(z) \right|^2 dz \leq C \int_{|z|> \alpha T} \left(\frac{1}{|z|^2}\wedge 1 \right) dz \leq \frac{C}{T}.$$
    Therefore, there exists $C>0$ such that for any $T>T_0$, 
    $$\left\| \Psi^{(T),+} - \rho^{+}\right\|_{\LL^2(0,1)} \leq \frac{C}{\sqrt{T}}.$$
    Besides, for $T\leq T_0$, we have
    $$\left\| \Psi^{(T),+} - \rho^{+} \right\|_{\LL^2(0,1)} \leq \left\| \Psi^{(T),+}\right\|_{\LL^2(0,1)}+ \left\|\rho^{+} \right\|_{\LL^2(0,1)} \leq e^{b_T T}\left\| \Psi^{(T),-}\right\|_{\LL^2(0,1)}+ \left\|\rho^{+} \right\|_{\LL^2(0,1)}.$$
    Note that 
    $$\left\| \Psi^{(T),-}\right\|_{\LL^2(0,1)} \leq \left\| \Psi^{T,-}\right\|_{\infty}.$$
    Now, remember Inequality \eqref{eq: maj_infty_psi_tilde_T}:
    $$\left\| \tilde{\Psi}^{T,-} \right\|_{\infty} \leq \left\| \tilde{\phi}^{T,-} \right\|_{\infty} \left(1 + \left\| \tilde{\Psi}^{T,-}\right\|_{1} \right).$$
    Since 
    $$\left\| \tilde{\phi}^{T,-} \right\|_{\infty} \leq \left\| \phi^{T,+} \right\|_{\infty} \leq 2 \left\| \phi \right\|_{\infty}\quad \text{and} \quad \left\| \tilde{\Psi}^{T,-}\right\|_{1} \leq \frac{\left\| \tilde{\phi}^{T,-} \right\|_{1}}{1-\left\| \tilde{\phi}^{T,-} \right\|_{1}}\leq \frac{1/a^+_T}{1-1/a^+_T}= T,$$
    we get 
    $$\left\| \Psi^{(T),-}\right\|_{\LL^2(0,1)} \leq C T \leq CT_0.$$
    Therefore, there exist $\tilde{C}>0$ and $C>0$ such that for $T\leq T_0$, we have
    $$\left\| \Psi^{(T),+} - \rho^{+} \right\|_{\LL^2(0,1)} \leq \frac{\tilde{C}}{\sqrt{T_0}} \leq \frac{C}{\sqrt{T}}.$$ 
\end{proof}

\subsubsection{Some inequalities on  \texorpdfstring{$d^{T,\natural}$}{}}
We gather here some inequalities on $\Psi^{(T),\natural}$ and $d^{T,\natural}$. In particular, we leverage Propositions \ref{prop: conv_psi_rho_-} and \ref{prop: conv_psi_rho_+} to establish regularity results on $\Psi^{(T),\natural}$, which allow us to remove the assumption $\sup_{T>1} \|\Psi^{(T),\natural}\|_{\infty} < +\infty$. Indeed, since our analysis of $d^{T,\natural}$ and $\Psi^{(T),\natural}$ is restricted to $[0,1]$, requiring the supremum norm to be uniformly bounded on $\RR_+$ is unnecessarily strong. Instead, we show directly that it remains uniformly bounded on $[0,1]$. This result is detailed in the following lemma:
\begin{lem}
    \label{lem: sup_norm_Psi}
    There exists $C>0$ such that for any $T>2$, we have $$i)\;\left\| \Psi^{(T),-}\right\|_{\LL^\infty (0,1)} \leq C, \quad ii)\;\left\| \Psi^{(T),0}\right\|_{\LL^\infty (0,1)} \leq C \quad \text{and} \quad iii)\;\left\| \Psi^{(T),+}\right\|_{\LL^\infty (0,1)} \leq C.$$
\end{lem}
\begin{proof}
    We start this proof by a remark that will be useful for our reasoning. Indeed, note that if we consider for $\natural \in \{-,0,+ \}$ the functions $\phi_{\exp}^{(T),\natural}: t\mapsto a^\natural_T \frac{1}{m}e^{-tT/m}$ then we get
    $$a^\natural_T \rho^\natural = \phi_{\exp}^{(T),\natural} + a^\natural_T T \rho^\natural * \phi_{\exp}^{(T),\natural}.  $$
    Moreover, we also have
    $$\Psi^{(T),\natural} = \phi^{(T),\natural}+ T \Psi^{(T),\natural}*\phi^{(T),\natural}.$$
    So we can write
    $$\Psi^{(T),\natural} - a^\natural_T \rho^\natural = \phi^{(T),\natural} -\phi_{\exp}^{(T),\natural} +T \left[   \Psi^{(T),\natural} - a^\natural_T\rho^\natural \right]\ast \phi^{(T),\natural}+ a^\natural_T T \rho^\natural \ast \left[\phi^{(T),\natural} - \phi_{\exp}^{(T),\natural}\right].$$
    This yields
    \begin{align*}
        \left\| \Psi^{(T),\natural} - a^\natural_T \rho^\natural \right\|_{\LL^\infty (0,1)} \leq & \left\| \phi^{(T),\natural} - \phi_{\exp}^{(T),\natural}\right\|_{\LL^\infty (0,1)}+T \left\| \left(   \Psi^{(T),\natural} -a^\natural_T \rho^\natural \right)\ast \phi^{(T),\natural}\right\|_{\LL^\infty (0,1)}\\
        &+a^\natural_T T \left\| \rho^\natural \ast \left(\phi^{(T),\natural} -\phi_{\exp}^{(T),\natural} \right)\right\|_{\LL^\infty (0,1)}.
    \end{align*}
    The first term of the right hand side of the inequality is bounded uniformly in $T$ by Assumption \ref{assump: phi}. In addition, using Cauchy-Schwarz Inequality and making a change of variable, we get
    \begin{align*}
        \left\| \Psi^{(T),\natural} - a^\natural_T \rho \right\|_{\LL^\infty (0,1)} \leq & C+ T \sup_{t\in [0,1]}  \int_0^t \left| \left(   \Psi^{(T),\natural} -a^\natural_T \rho^\natural \right)(t-s) \phi^{(T),\natural}(s)\right| ds \\
        &+ a^\natural_T T \left\| \rho^\natural \right\|_{\LL^{\infty} (0,1)} \int_0^1 \left| \left(\phi^{(T),\natural} - \phi_{\exp}^{(T),\natural} \right)(s) \right| ds\\
        \leq& C+\sqrt{T}\left\| \Psi^{(T),\natural} -a^\natural_T \rho^\natural \right\|_{\LL^2(0,1)}\left\| \phi \right\|_{2} \\
        &+ C T\left\| \rho^\natural \right\|_{\LL^{\infty} (0,1)} \left\| \phi^{(T),\natural} -\phi_{\exp}^{(T),\natural} \right\|_{\LL^1 (0,1)} .
    \end{align*}
    Moreover, by Assumption \ref{assump: phi}, we obtain
        \begin{align*}
        \left\| \Psi^{(T),\natural} - a^\natural_T \rho^\natural \right\|_{\LL^\infty (0,1)} \leq& C+ C \sqrt{T}\left\| \Psi^{(T),\natural} -a^\natural_T \rho^\natural \right\|_{\LL^2(0,1)}\end{align*} 
    Therefore, using Proposition \ref{prop: conv_psi_rho_-}, \ref{prop: conv_psi_rho_0} or \ref{prop: conv_psi_rho_+} according the considered case, we have for $T$ large enough
            \begin{align*}
        \left\| \Psi^{(T),\natural} - a^\natural_T \rho^\natural \right\|_{\LL^\infty (0,1)} \leq& C .\end{align*}
    We conclude the proof using the triangular inequality.
\end{proof}

\begin{lem}
\label{lem: result d'}
    There exists $C>0$ such that for any $T>2$, $$ \left\| (d^{T,\natural})'\right\|_{\LL^\infty(0,1)} \leq C T \quad \text{and} \quad \left\| (d^{T,\natural})'\right\|_{ \LL^2(0,1)} \leq C \sqrt{T},$$
    where $\natural \in \{ -,0,+\}$.
\end{lem}

\begin{proof}
    For any $t\in [0,1]$, we have:
    $$\left(\Psi^{(T),\natural}\right)'(t) = T \left(a^\natural_T\phi'(Tt)+\left(\phi^{(T),\natural}\right)'\ast \Psi^{(T),\natural}(t) \right).$$
    Hence, 
    $$\left\| \left(\Psi^{(T),\natural}\right)'\right\|_{\LL^\infty(0,1)} \leq C T \left( \left\| \phi' \right\|_{\infty} + \left\| \Psi^{(T),\natural}\right\|_{\LL^\infty(0,1)} \left\| \phi' \right\|_1 \right).$$
    By Assumption \ref{assump: phi} and Lemma \ref{lem: sup_norm_Psi}, we get:
    $$\left\| \left(\Psi^{(T),\natural}\right)'\right\|_{\LL^\infty(0,1)} \leq CT.$$
    Therefore, we can conclude that $\left\| \left( d^{T,\natural}\right)'\right\|_{\LL^\infty(0,1)} \leq CT$.\\
    The next result is a direct consequence of this one.
\end{proof}

Moreover, the two next lemmas extend Lemma 4.7 in \cite{jaisson_limit_2015}. Indeed, we upgrade their upper bound by proving that the critical case also verified the property.
\begin{lem}
\label{lem: maj_diff_rho}
    There exists $C>0$ such that for any $(s,t)\in [0,1]^2$ such that $s<t$, 
    $$\int_0^s \left| \rho^\natural(t-u) - \rho^\natural(s-u)\right|^2 du \leq C(t-s),$$
    where $\natural \in \{ -,0,+\}$.
\end{lem}
\begin{proof}
    The proof holds in the study of $\rho^\natural$.
\end{proof}

\begin{lem}
\label{lem: maj_diff_psi}
    There exists $C>0$ such that for any $T>2$ and any $(s,t)\in [0,1]^2$ such that $s<t$, 
    $$\int_0^s \left| \Psi^{(T),\natural}(t-u) - \Psi^{(T),\natural}(s-u)\right|^2 du \leq C(t-s),$$
    where $\natural \in \{ -,0,+\}$.
\end{lem}
\begin{proof}
 For $(s,t)\in [0,1]^2$ such that $s<t$, we have
  \begin{align*}
     \int_0^s \left[ \Psi^{(T),\natural} (t-u) - \Psi^{(T),\natural}(s-u)\right]^2 du \leq& C \int_0^s \left| d^{T,\natural}(t-u)-d^{T,\natural}(s-u) \right|^2 du \\
     &+ C\int_0^s \left|\rho^\natural(t-u) -\rho^\natural(s-u)\right|^2 du.
 \end{align*}
 Since $d^{T,\natural}$ is differentiable, we have
     \begin{align*}
     \int_0^{s} \left|d^{T,\natural}(t-u) -d^{T,\natural}(s-u)\right|^2 du &= \int_0^s \left|d^{T,\natural}(t-u) -d^{T,\natural}(s-u)\right| \left|\int_s^t \left(d^{T,\natural} \right)'(r-u) dr \right| du\\
     &\leq \int_s^t \int_0^s \left|d^{T,\natural}(t-u) -d^{T,\natural}(s-u)\right| \left|\left(d^{T,\natural} \right)'(r-u)\right| du dr.
 \end{align*}
 Then, by Cauchy-Schwarz Inequality and Jensen Inequality, we get:
      \begin{align*}
     \int_0^{s} \left|d^{T,\natural}(t-u) -d^{T,\natural}(s-u)\right|^2 du &\leq \int_s^t \sqrt{\int_0^s \left[d^{T,\natural}(t-u) -d^{T,\natural}(s-u)\right]^2 du \int_0^s \left(d^{T,\natural}\right)'(r-u)^2 du }dr.
 \end{align*}
Since $\left\|d^{T,\natural}\right\|_{L^2(0,1)}^2\leq \frac{C}{T}$ (Propositions \ref{prop: conv_psi_rho_-}, \ref{prop: conv_psi_rho_0} or \ref{prop: conv_psi_rho_+} ) and $\left\|\left(d^{T,\natural}\right)'\right\|_{L^2(0,1)}^2 \leq CT$ (Lemma \ref{lem: result d'}), we obtain:
      \begin{equation}
      \label{eq: maj_diff_d^T}
     \int_0^{s} \left|d^{T,\natural}(t-u) -d^{T,\natural}(s-u)\right|^2 du \leq C \int_s^t \frac{1}{T} T du = C(t-s).
    \end{equation}
 Therefore, combining Lemma \ref{lem: maj_diff_rho} and \eqref{eq: maj_diff_d^T} we finally have:
 \begin{align*}
     \int_0^s \left[ \Psi^{(T),\natural} (t-u) - \Psi^{(T),\natural}(s-u)\right]^2 du \leq C(t-s).
 \end{align*}
\end{proof}

\subsection{Results on \texorpdfstring{$\Lambda^{T,\natural}$}{}}
\label{subs: TL for both thm}

In order to apply the Theorem \ref{thm: maj_sup_int_NTW}, we need to verify if $\Lambda^{T,\natural}$ satisfies the different assumptions for any $\natural \in \{-,0,+ \}$. In particular, we need to bound the second moment of $\Lambda^{T,\natural}$ from above and this uniformly in $T$. We also decide to prove that $\Lambda^{T,\natural}$ satisfies some Holdër condition. These are resume in the next proposition where $\natural \in \{ -,0,+\}$ is fixed.
\begin{prop}
\label{prop: regularite_lambda}
    There exists $C>0$ such that for any $T>2$ and $(t,s)\in [0,1]^2$ such that $s<t$, 
    $$\E{ \left| \Lambda^{T,\natural}_t\right|^2} <C \quad \text{and}\quad \E{\left|\Lambda^{T,\natural}_t - \Lambda^{T,\natural}_s \right|^2} \leq C(t-s).$$
\end{prop}

\begin{proof}
    For $t\in [0,1]$, we have:
    $$\left| \Lambda^{T,\natural}_t\right|^2 \leq C\frac{\mu}{T^2} + C \mu^2 \left( \int_0^t \Psi^{(T),\natural}(u) du \right)^2 + C\left( \iint_{(0,t)\times \RR_+} \Psi^{(T),\natural}(t-s) \1_{\theta \leq \Lambda^{T,\natural}_s} \Ntil^T(ds,d\theta)\right)^2 .$$
    Besides, Lemma \ref{lem: sup_norm_Psi} allows us to take the expectation and yields:
\begin{align*}
    \hspace*{-0,5cm}\E{\left| \Lambda^{T,\natural}_t\right|^2} &\leq C\frac{\mu^2}{T^2} + C\left\| \Psi^{(T),\natural}\right\|_{\LL^\infty(0,1)}^2 \mu^2 + C\E{\left( \iint_{(0,t)\times \RR_+} \Psi^{(T),\natural}(t-s) \1_{\theta \leq \Lambda^{T,\natural}_s} \Ntil^T(ds,d\theta)\right)^2}\\
    &\leq C\frac{\mu^2}{T^2} + C\left\| \Psi^{(T),\natural}\right\|_{\LL^\infty(0,1)}^2 \mu^2 + C\left\| \Psi^{(T),\natural}\right\|_{\LL^\infty(0,1)}^2  \int_0^t \E{ \Lambda^{T,\natural}_s} ds .
\end{align*}
Note that since $\phi^{T,\natural}$ is locally square integrable, Theorem 8 of \cite{delattre_hawkes_2016} ensures the existence of $\E{\Lambda^{T,\natural}_{\cdot}}$. Moreover, for any $t\in [0,1]$, 
$$\E{\Lambda_t^{T,\natural}} = \frac{\mu}{T} + \int_0^t \Psi^{(T),\natural}(t) dt \leq \mu + \left\|\Psi^{(T),\natural} \right\|_{\LL^\infty(0,1)} \leq C.$$
Thus, we obtain that
$$\sup_{t\in [0,1]}\E{\left| \Lambda^{T,\natural}_t\right|^2} \leq C. $$

On the other hand, for $0\leq s<t\leq 1$, we have:
\begin{align*}
    \Lambda^{T,\natural}_t - \Lambda^{T,\natural}_s =& \mu \int_s^t \Psi^{(T),\natural}(u) du + \iint_{[s,t)\times \RR_+} \Psi^{(T),\natural}(t-u) \1_{\theta \leq \Lambda^{T,\natural}_u} \Ntil^T(du,d\theta) \\
    &+ \iint_{(0,s)\times \RR_+} \left( \Psi^{(T),\natural}(t-u)-\Psi^{(T),\natural}(s-u)\right) \1_{\theta \leq \Lambda^{T,\natural}_u} \Ntil^T(du,d\theta).
\end{align*}
This yields
\begin{align*}
    \E{\left|\Lambda^{T,\natural}_t - \Lambda^{T,\natural}_s\right|^2} \leq& \mu^2 C \left| \int_s^t\Psi^{(T),\natural}(u) du \right|^2+ C\E{\left|\iint_{[s,t)\times \RR_+} \Psi^{(T),\natural}(t-u) \1_{\theta \leq \Lambda^{T,\natural}_u} \Ntil^T(du,d\theta)\right|^2} \\
    &+ C \E{ \left|\iint_{(0,s)\times \RR_+} \left( \Psi^{(T),\natural}(t-u)-\Psi^{(T)}(s-u)\right) \1_{\theta \leq \Lambda^{T,\natural}_u} \Ntil^T(du,d\theta) \right|^2}\\
    \leq& \mu^2 C \left|\int_s^t \Psi^{(T),\natural}(u) du \right|^2+ C \int_s^t \left|\Psi^{(T),\natural}(t-u)\right|^2 \E{\Lambda^{T,\natural}_u }du \\
    &+ C \int_0^s \left|\Psi^{(T),\natural}(t-u)-\Psi^{(T),\natural}(s-u)\right|^2 \E{\Lambda^{T,\natural}_u }du.
\end{align*}
Since $\left\| \Psi^{(T),\natural}\right\|_{\LL^\infty(0,1)} \leq C$ and $\sup_{u\in [0,1]}\E{\Lambda^{T,\natural}_u }\leq C$, we have
\begin{align*}
    \E{\left|\Lambda^{T,\natural}_t - \Lambda^{T,\natural}_s\right|^2} \leq& C(t-s) + C \int_0^s \left|\Psi^{(T),\natural}(t-u)-\Psi^{(T),\natural}(s-u)\right|^2 du \leq C(t-s).
\end{align*}
Here we concluded with Lemma \ref{lem: maj_diff_psi}.
\end{proof}

\paragraph{Discretized processes}~\\
For $t\in [0,1]$, we recall the discretized version of our process $\Lambda^T$:
\begin{equation*}
    \bar{\Lambda}^{T,\natural}_t := \sum_{i=0}^{k-1} \1_{I_{i,k}}(t) \Lambda^{T,\natural}_{t_i^k}
\end{equation*}
Then, we have the following proposition:
\begin{prop}
\label{prop: discret_lambda_rho}
    There exists $C>0$ such that for any $T>2$, 
    $$\sup_{t\in[0,1]}\E{\left|\Lambda^{T,\natural}_t - \bar{\Lambda}^{T,\natural}_t \right|^2} \leq \frac{C}{k}.$$
\end{prop}
\begin{proof}
    Let $t\in [0,1]$. Then, we have
    $$\E{\left|\Lambda^{T,\natural}_t - \bar{\Lambda}^{T,\natural}_t \right|^2} = \E{\sum_{i=0}^{k-1} \1_{I_{i,k}}(t)\left|\Lambda^{T,\natural}_t - \Lambda^{T,\natural}_{t_i^k} \right|^2 }=\sum_{i=0}^{k-1} \1_{I_{i,k}}(t)\E{\left|\Lambda^{T,\natural}_t - \Lambda^{T,\natural}_{t_i^k} \right|^2 }.$$
    Thus, by Proposition \ref{prop: regularite_lambda}, we get
    $$\E{\left|\Lambda^{T,\natural}_t - \bar{\Lambda}^{T,\natural}_t \right|^2} \leq C \sum_{i=0}^{k-1} \1_{I_{i,k}}(t)\left|t - t_i^k \right| \leq \frac{C}{k}.$$
\end{proof}

\subsection{Tools for the Proof of Theorem \ref{th: NUHP to X}}
\label{subs: tools for theorem NUHP to X}
In this section, we want to control $\E{\sup_{t\in [0,1]}\left|R^T_t\right|^2}$. To do so, let us recall some notation, previously introduced in \eqref{eq: Delta^T avec reste}:
\begin{align*}
    Y^T_t &= \iint_{(0,t)\times \RR_+} d^T(t-s) \1_{\theta \leq \Lambda^T_s} \Ntil^T(ds,d\theta), \\
    M^{\Lambda^T, \Ntil^T}_t &= \iint_{(0,t]\times \RR_+} \rho(-s) \left[\1_{\theta \leq \Lambda^T_s} - \1_{\theta \leq \bar{\Lambda}^T_s} \right] \Ntil^T(ds,d\theta),\\
    M^{\Ntil^T-W}_t &= \iint_{(0,t]\times \RR_+} \rho(-s)\1_{\theta \leq \bar{\Lambda}^T_s} (\Ntil^T-W)(ds,d\theta).
\end{align*}
Also, note that since this section concern the Proof of Theorem \ref{th: NUHP to X}, we do not make use of the index $\natural$ as indicated in the begin of the proof.

\begin{prop}
\label{prop: maj_YT}
    There exists $C>0$ such that for any $T>2$,
    $$\E{\sup_{t\in [0,1]} \left| Y^T_{t} \right|^4} \leq \frac{C}{T}$$
\end{prop}

\begin{proof}
    Since $d^T$ is differentiable, we can write for any $(s,t) \in [0,1]^2$ such that $s<t$, 
    $$d^T(t-s) = d^T(0) + \int_s^t \left( d^T\right)'(u-s) du.$$
    Hence, combining this with Fubini's Theorem, we get
    \begin{equation} \label{eq: reecriture YT}
      Y^T_t = d^T(0) \int_{(0,t)\times \RR_+} \1_{\theta \leq \Lambda^T_s} \Ntil^T(ds,d\theta) + \int_0^t \left( \int_{(0,u)\times \RR_+} \left( d^T\right)'(u-s) \1_{\theta \leq \Lambda^T_s} \Ntil^T(ds,d\theta) \right)du .
\end{equation}
We now define the càd-làg semimartingale $\Ycal^T$ by
    \begin{equation*}
        \Ycal^{T}_t = d^T(0) \int_{(0,t]\times \RR_+} \1_{\theta \leq \Lambda^T_s} \Ntil^T(ds,d\theta) + \int_0^t \left( \int_{(0,u)\times \RR_+} \left( d^T\right)'(u-s) \1_{\theta \leq \Lambda^T_s} \Ntil^T(ds,d\theta) \right)du .
    \end{equation*}
We thereupon deduce that $Y^T_{t} = \Ycal^{T}_{t^-}$ and we get that 
    $$\sup_{t \in [0,1]}\left|Y^T_t\right| \leq  \sup_{t \in [0,1]}\left|\Ycal^{T}_t \right|. $$

We now work on $\Ycal^T$ and we denote by $(a_i)_{i=0,\dots, N}$ a regular subdivision of $[0,1]$. Then, we have:
    \begin{align*}
        \E{\sup_{t\in [0,1]} \left| \Ycal^T_t \right|^4} &\leq \E{\max_{i=0, \dots, N-1} \sup_{t\in[ a_i , a_{i+1})} \left| \Ycal^T_t - \Ycal^T_{a_i}\right|^4 } + \E{\max_{i=1, \dots, N} \left| \Ycal^T_{a_i}\right|^4} \\
        &\leq \sum_{i=0}^{N-1} \E{\sup_{t\in [ a_i , a_{i+1})} \left| \Ycal^T_t - \Ycal^T_{a_i}\right|^4 } + \sum_{i=1}^{N}\E{\left| \Ycal^T_{a_i}\right|^4}.
    \end{align*}
    We split the proof in two parts:
    \paragraph{Control of \texorpdfstring{$\sum_{i=0}^{N-1}\E{\sup_{t\in [a_i , a_{i+1}) } \left| \Ycal^T_t - \Ycal^T_{a_i}\right|^4 }$}{}}.~\\
    
    We first make use of \eqref{eq: reecriture YT} and the fact that for $T$ large enough $\left\|d^T\right\|_{\LL^\infty(0,1)} \leq C$ to obtain
    \begin{align*}
        \hspace*{-1cm}\sum_{i=0}^{N-1}\E{\sup_{t\in [a_i , a_{i+1}) }
        \left| \Ycal^T_t - \Ycal^T_{a_i}\right|^4 } \leq& C \sum_{i=0}^{N-1} \E{\sup_{t\in [a_i, a_{i+1})} \left| \iint_{[a_i,t] \times \RR_+} \1_{\theta \leq \Lambda^T_s} \Ntil^T(ds,d\theta) \right|^4}\\
        & + \sum_{i=0}^{N-1} \E{\sup_{t\in [a_i, a_{i+1})} \left| \int_{a_i}^{t} \left( \iint_{(0,u)\times \RR_+} \left( d^T\right)'(u-s)\1_{\theta \leq \Lambda^T_s} \Ntil^T(ds,d\theta) \right) du \right|^4} .
    \end{align*}
    Lemma \ref{lem-outil-tilde-N-T}, applied with $u_{(s,\theta)} = \1_{(a_i,a_{i+1}]}(s) \1_{\theta \leq \Lambda^T_s}$, yields
        \begin{align*}
        \hspace*{-1cm}\sum_{i=0}^{N-1}\E{\sup_{t\in[ a_i , a_{i+1})} \left| \Ycal^T_t - \Ycal^T_{a_i}\right|^4 } \leq& C \sum_{i=0}^{N-1} \E{\left(\int_{a_i}^{a_{i+1}} \Lambda^T_s ds \right)^2 + \frac{1}{T^2}\int_{a_i}^{a_{i+1}} \Lambda^T_s ds }\\
        & + \sum_{i=0}^{N-1} \E{\left| \int_{a_i}^{a_{i+1}} \left| \iint_{(0,u)\times \RR_+}\left( d^T\right)'(u-s)\1_{\theta \leq \Lambda^T_s} \Ntil^T(ds,d\theta) \right| du \right|^4}.
    \end{align*}
    Thanks to Jensen's inequality, we compute 
        \begin{align*}
        \hspace*{-1cm}\sum_{i=0}^{N-1}\E{\sup_{t\in[ a_i , a_{i+1})} \left| \Ycal^T_t - \Ycal^T_{a_i}\right|^4 } \leq& C \sum_{i=0}^{N-1} \frac{1}{N} \int_{a_i}^{a_{i+1}} \E{\left|\Lambda^T_s\right|^2} ds  + \frac{1}{T^2}\int_{a_i}^{a_{i+1}} \E{\Lambda^T_s} ds \\
        & + \sum_{i=0}^{N-1} \frac{1}{N^3}\E{\int_{a_i}^{a_{i+1}} \left| \iint_{(0,u)\times \RR_+} \left( d^T\right)'(u-s)\1_{\theta \leq \Lambda^T_s} \Ntil^T(ds,d\theta) \right|^4 du }.
    \end{align*}
    Thus,
    \begin{align*}
        \hspace*{-1cm}\sum_{i=0}^{N-1}\E{\sup_{t\in[ a_i , a_{i+1})} \left| \Ycal^T_t - \Ycal^T_{a_i}\right|^4 } \leq& \frac{C}{N}  + \frac{C}{T^2} + \frac{1}{N^3}\int_0^1 \E{\left| \iint_{(0,u)\times \RR_+} \left( d^T\right)'(u-s)\1_{\theta \leq \Lambda^T_s} \Ntil^T(ds,d\theta) \right|^4 }du.
    \end{align*}
    Moreover, by Lemma \ref{lem-outil-tilde-N-T}, we have for any $u\in [0,1]$:
    \begin{align*}
        \E{\left| \int_0^u \int_{\RR_+} \left( d^T\right)'(u-s)\1_{\theta \leq \Lambda^T_s} \Ntil^T(ds,d\theta) \right|^4 } \leq& C\E{ \left(\int_0^u \left| \left( d^T\right)'(u-s)\right|^2 \Lambda^T_s ds\right)^2} \\
        &+ \frac{C}{T^2}\int_0^u \left| \left( d^T\right)'(u-s)\right|^4 \E{\Lambda^T_s}  ds.
    \end{align*}
    However, by the Cauchy-Schwarz inequality, we deduce that for any $u\in [0,1]$,
    \begin{align*}
        \E{\left(\int_0^u \left| \left( d^T\right)'(u-s)\right|^2 \Lambda^T_s ds\right)^2} &\leq \left\| \left( d^T\right)'\right\|_{\LL^2(0,1)}^2 \int_0^u \left| \left( d^T\right)'(u-s)\right|^2 \E{\left|\Lambda^T_s\right|^2} ds \\
        &\leq \sup_{s\in [0,1]} \E{\left|\Lambda^T_s\right|^2} \left\| \left( d^T\right)'\right\|_{\LL^2(0,1)}^4.
    \end{align*}
    Since $\sup_{s\in [0,1]} \E{\left|\Lambda^T_s\right|^2}\leq C$, $\left\|\left(d^T\right)'\right\|_{\LL^\infty (0,1)}\leq CT$ and $\left\|\left(d^T\right)'\right\|_{\LL^2(0,1)} ^2\leq CT$, we get
        \begin{align*}
        \E{\left| \iint_{(0,u)\times \RR_+} \left( d^T\right)'(u-s)\1_{\theta \leq \Lambda^T_s} \Ntil^T(ds,d\theta) \right|^4 } \leq& CT^{2}.
    \end{align*}
    Therefore, we can conclude that
    \begin{align*}
        \sum_{i=0}^{N-1}\E{\sup_{t\in( a_i , a_{i+1}]} \left| \Ycal^T_t - \Ycal^T_{a_i}\right|^4 } \leq& C\left( \frac{1}{N} + \frac{1}{T^2} + \frac{T^2}{N^3}\right).
    \end{align*}
    \paragraph{Control of \texorpdfstring{$\sum_{i=1}^{N}\E{\left| \Ycal^T_{a_i}\right|^4 }$}{}.}~\\
    In this paragraph, we will employ the first writing of $\Ycal^T$, that is
    $$\Ycal^T_t = \iint_{(0,t)\times\RR_+}d^T(t-s)\1_{\theta \leq \Lambda_s^T}\Ntil^T(ds,d\theta), \quad t\in [0,1].$$
    Hence, by Lemma \ref{lem-outil-tilde-N-T}, we obtain
    \begin{align*}
        \E{\left| \Ycal^T_{a_i}\right|^4 } &\leq C\E{\left( \int_0^{a_i} \left|d^T(a_i-s)\right|^2 \Lambda^T_s  ds \right)^2} + \frac{C}{T^2}\int_0^{a_i} \left|d^T(a_i-s)\right|^2 \E{\Lambda^T_s} ds
    \end{align*}
    and so, by the same computations as before, we get:
    $$\E{\left( \int_0^{a_i} \left|d^T(a_i-s)\right|^2 \Lambda^T_s  ds \right)^2} \leq \sup_{s\in [0,1]} \E{\left|\Lambda^T_s\right|^2} \left\| d^T\right\|_{\LL^2(0,1)}^4 \leq \frac{C}{T^2}.$$
    Thus we obtain that 
    $$ \sum_{i=1}^{N} \E{\left| \Ycal^T_{a_i}\right|^4 } \leq \sum_{i=1}^{N}\frac{C}{T^2} \leq C\frac{N}{T^2}.$$
    Finally, we take $N=T$ and we conclude that 
    $$\E{\sup_{t\in[0,1]} \left|Y^T_t \right|^4}\leq \E{\sup_{t\in[0,1]} \left|\Ycal^T_t \right|^4} \leq \frac{C}{T}.$$
\end{proof}

\begin{prop}
\label{prop: maj_RT}
    There exists $C>0$ such that for any $T>2$, 
    $$\E{\sup_{t\in [0,1]} \left|R^T_t \right|^2} \leq C\left( \frac{1}{\sqrt{k}} + \frac{k^2}{T^2} + \frac{1}{T} \right).$$
\end{prop}

\begin{proof}
    We start this by using the triangular inequality and Jensen inequality to get:
    \begin{align*}
        \E{\sup_{t\in [0,1]} \left|R^T_t \right|^2} \leq& C \int_0^1 \left|d^T(u)\right|^2 du+ C\sqrt{\E{\sup_{t\in[0,1]} \left| Y^T_t\right|^4} }\\
        &+ C\E{\sup_{t\in[0,1]} \left| M^{\Lambda^T, \Ntil^T}_t\right|^2}+ C\E{\sup_{t\in[0,1]} \left| M^{\Ntil^T-W}_t\right|^2}.
    \end{align*}
    The first term can be control thanks to Propositions \ref{prop: conv_psi_rho_-}, \ref{prop: conv_psi_rho_0} and \ref{prop: conv_psi_rho_+} which yields
    $$\E{\sup_{t\in [0,1]} \left|R^T_t \right|^2} \leq \frac{C}{T}+ C\sqrt{\E{\sup_{t\in[0,1]} \left| Y^T_t\right|^4} }+ C\E{\sup_{t\in[0,1]} \left| M^{\Lambda^T, \Ntil^T}_t\right|^2}+ C\E{\sup_{t\in[0,1]} \left| M^{\Ntil^T-W}_t\right|^2}.$$
   By Propositions \ref{prop: maj_YT} and \ref{prop:maj-intw-nT}, we have:
   $$\E{\sup_{t\in [0,1]} \left|R^T_t \right|^2} \leq C\E{\sup_{t\in[0,1]} \left| M^{\Lambda^T, \Ntil^T}_t\right|^2}+ C\left( \frac{1}{\sqrt{T}}+\frac{1}{k} + \frac{k^2}{T^2} \right).$$
   Moreover, Doob's inequality yields
   \begin{align*}
       \E{\sup_{t\in [0,1]} \left|R^T_t \right|^2} &\leq C\E{\left| M^{\Lambda^T, \Ntil^T}_1\right|^2}+ C\left( \frac{1}{\sqrt{T}}+\frac{1}{k} + \frac{k^2}{T^2} \right)\\
       &\leq \int_0^1 \rho(-s)^2\E{\left| \Lambda^T_s -\bar{\Lambda}^T_s \right|} ds +C\left( \frac{1}{\sqrt{T}}+\frac{1}{k} + \frac{k^2}{T^2} \right).
   \end{align*}
   Finally, by Proposition \ref{prop: discret_lambda_rho} and since $\left\|\rho \right\|_{\LL^\infty(-1,1)} \leq C $, we get:
      \begin{align*}
       \E{\sup_{t\in [0,1]} \left|R^T_t \right|^2} &\leq C\left( \frac{1}{\sqrt{T}}+\frac{1}{\sqrt{k}} + \frac{k^2}{T^2} \right).
   \end{align*}
\end{proof}

\paragraph{Acknowledgment}: This work received support from the University Research School EUR-MINT
(State support managed by the National Research Agency for Future Investments
 program bearing the reference ANR-18-EURE-0023).

\begin{sloppypar}
\printbibliography
\end{sloppypar}
\end{document}